\documentclass[reqno]{amsart}
\usepackage{amssymb,amscd,verbatim, amsthm,graphicx, color}
\usepackage{fancybox}
\usepackage{longtable}

\newcommand \fk[1]{{{\mathfrak #1}}}
\newcommand \C[1]{{\mathcal #1}}

\newcommand \wti[1]{{\widetilde {#1}}}

\newcommand\fg{\mathfrak g}

\newcommand \bC{{\mathbb C}}

\newcommand \bH{{\mathbb H}}
\newcommand \bR{{\mathbb R}}
\newcommand \bZ{{\mathbb Z}}

\newcommand\CB{{\C B}}
\newcommand\CH{{\C H}}

\newcommand\CS{{\C S}}

\newcommand\cf{{\it cf.~ }}

\newcommand\ep{{\epsilon}}

\newcommand\om{{\omega}}
\newcommand\al{{\alpha}}

\newtheorem{theorem}{Theorem}[section]

\newtheorem{corollary}[theorem]{Corollary}
\newtheorem{lemma}[theorem]{Lemma}
\newtheorem{proposition}[theorem]{Proposition}

\theoremstyle{definition}
\newtheorem{definition}[theorem]{Definition}
\newtheorem{remark}[theorem]{Remark}
\newtheorem{example}[theorem]{Example}
\newtheorem{notation}[theorem]{Notation}

\newcommand\Hom{\operatorname{Hom}}
\newcommand\End{\operatorname{End}}
\newcommand\Ext{\operatorname{Ext}}
\newcommand\Ind{\operatorname{Ind}}
\newcommand\ind{\operatorname{ind}}
\newcommand\res{\operatorname{res}}
\newcommand\rank{\operatorname{rank}}

\newcommand\im{\operatorname{im}}
\newcommand\supp{\operatorname{supp}}
\newcommand\Sg{\operatorname{Sg}}
\newcommand\Det{\operatorname{det}}

\newcommand\triv{\mathsf{triv}}
\newcommand\sgn{\mathsf{sgn}}

\newcommand\gen{\mathsf{gen}}
\newcommand\ds{\mathsf{DS}}\newcommand\ET{\mathsf{ET}}

\newcommand\af{\mathsf{aff}}
\newcommand\Rep{\mathsf{Rep}}
\newcommand\fd{\mathsf{fd}}
\newcommand\el{\mathsf{ell}}
\newcommand\EP{\mathsf{EP}}
\newcommand\even{\mathsf{even}}
\newcommand\odd{\mathsf{odd}}
\newcommand\Irr{\mathsf{Irr}}
\newcommand\DS{\mathsf{DS}}
\newcommand\rr{\mathsf{r}}

\newcommand\Res{\mathsf{Res}}
\newcommand\lin{\mathsf{lin}}
\newcommand\cc{\mathsf{cc}}
\newcommand\disc{\mathsf{disc}}
\newcommand\red{\mathsf{red}}

\newcommand\Pin{\mathsf{Pin}}
\newcommand\Spec{\mathsf{Spec}}

\newlength{\tabwidth}
\newlength{\tabheight}
\newlength{\tabrule}
\newlength{\tabwidthx}
\newlength{\tabheightx}

\def\gentabbox#1#2#3#4{\vbox to \tabheight{\setlength{\tabrule}{#3}%
  \setlength{\tabwidthx}{#1\tabwidth}\addtolength{\tabwidthx}{\tabrule}%

\setlength{\tabheightx}{#2\tabheight}\addtolength{\tabheightx}{-\tabheight}%
  \hbox to #1\tabwidth{%
    \hspace{-0.5\tabrule}\rule{\tabrule}{#2\tabheight}\hspace{-\tabrule}%
    \vbox to #2\tabheight{\hsize=\tabwidthx%
      \vspace{-0.5\tabrule}\hrule width\tabwidthx height\tabrule%
      \vspace{-0.5\tabrule}\vfil%
      \hbox to \tabwidthx{\hss#4\hss}%
        \vfil\vspace{-0.5\tabrule}%
      \hrule width\tabwidthx height\tabrule\vspace{-0.5\tabrule}}%
    \hspace{-\tabrule}\rule{\tabrule}{#2\tabheight}\hspace{-0.5\tabrule}}%
  \vspace{-\tabheightx}}}
\def\genblankbox#1#2{\vbox to \tabheight{\vfil\hbox to
#1\tabwidth{\hfil}}}
\def\tabbox#1#2#3{\gentabbox{#1}{#2}{0.4pt}{\strut #3}}

\catcode`\:=13 \catcode`\.=13 \catcode`\;=13 
\catcode`\>=13 \catcode`\^=13
\def:#1\\{\hbox{$#1$}}
\def.#1{\tabbox{1}{1}{$#1$}}
\def>#1{\tabbox{2}{1}{$#1$}}
\def^#1{\tabbox{1}{2}{$#1$}}
\def;{\genblankbox{1}{1}\relax}
\catcode`\:=12 \catcode`\.=12 \catcode`\;=12 
\catcode`\>=12 \catcode`\^=7

\newenvironment{tableau}{\bgroup\catcode`\:=13 \catcode`\.=13
  \catcode`\;=13 \catcode`\>=13 \catcode`\^=13
  \setlength{\tabheight}{3ex}\setlength{\tabwidth}{3ex}%
  \def\b##1##2##3{\gentabbox{##1}{##2}{1.2pt}{\vbox{##3}}}%
  \def\n##1##2##3{\gentabbox{##1}{##2}{0.4pt}{\vbox{##3}}}%
  \vbox\bgroup\offinterlineskip}{\egroup\egroup}

\numberwithin{equation}{subsection}

\begin{document}

\title[Algebraic and analytic Dirac induction]{Algebraic and analytic Dirac induction for graded affine Hecke algebras}

\author{Dan Ciubotaru}
        \address[D. Ciubotaru]{Department of Mathematics\\ University of
          Utah\\ Salt Lake City, UT 84112, USA}
        \email{ciubo@math.utah.edu}

\author{Eric M.~Opdam}
\address[E. Opdam]{Korteweg-de Vries Institute for Mathematics\\Universiteit van Amsterdam\\Science Park 904\\ 1098 XH Amsterdam, The Netherlands}
\email{e.m.opdam@uva.nl} 

\author{Peter E.~Trapa}
        \address[P. Trapa]{Department of Mathematics\\ University of
          Utah\\ Salt Lake City, UT 84112, USA}
        \email{ptrapa@math.utah.edu}

\thanks{The authors thank the Max Planck Institute for Mathematics in
  Bonn for the hospitality and support while part of this research was
  carried out, and  the organizers of  ``Analysis
  on Lie groups'' for their invitation to participate in the
  program. The authors also thank the referee of the paper for his careful reading of the manuscript and for many useful comments. This research was supported by  NSF-DMS 0968065 and NSA-AMS 081022 (D. Ciubotaru), ERC-advanced grant no. 268105 (E. Opdam), 
  NSF-DMS 0968275 and 0968060 (P. Trapa).
    }

\begin{abstract}

We define the algebraic Dirac induction map $\Ind_D$ for graded affine Hecke algebras.
The map $\Ind_D$ is a Hecke algebra analog of the explicit realization of the 
Baum-Connes assembly map in the $K$-theory of the reduced $C^*$-algebra of a 
real reductive group using Dirac operators.
The definition of $\Ind_D$ is uniform over the parameter space of the graded 
affine Hecke algebra. 
We show that the map $\Ind_D$ defines an isometric isomorphism 
from the space of elliptic characters of the Weyl group (relative to its reflection representation) 
to the space of elliptic characters of the graded affine Hecke algebra. 
We also study a related analytically defined global elliptic Dirac operator between unitary 
representations of the graded affine Hecke algebra which are realized in the spaces of sections of 
vector bundles associated to certain representations of the pin cover of the Weyl group.  
In this way we realize all irreducible discrete series modules of the
Hecke algebra in the kernels (and indices) of such analytic Dirac
operators.
This can be viewed as a graded affine Hecke algebra analogue of the construction of the discrete series representations of semisimple Lie groups due to Parthasarathy and Atiyah-Schmid.
\end{abstract}

\maketitle

\setcounter{tocdepth}{1}

\begin{small}
\tableofcontents
\end{small}

\section{Introduction}\label{s:1}

\subsection{}\label{s:1.1}
Graded affine Hecke algebras were
introduced independently by Lusztig \cite{L}, as a tool to study
representations of reductive $p$-adic groups and Iwahori-Hecke algebras, and by Drinfeld
\cite{Dr1,Dr2}, in connection with the
study of yangians, a certain class of quantum groups. 

In the seminal paper \cite{B}, Borel proved that the category of smooth (admissible) representations of a reductive $p$-adic group, which are generated by vectors fixed under an Iwahori subgroup, is naturally equivalent with the category of (finite-dimensional) modules for the Iwahori-Hecke algebra $\C H$, i.e., the convolution algebra of compactly supported complex functions on the group, which are left and right invariant under translations by elements in the Iwahori subgroup. The graded Hecke algebras $\bH$ (Definition
\ref{d:graded}) are certain degenerations of Iwahori-Hecke algebras, and one can recover much of the representation theory of $\C H$ from $\bH$. Moreover, via Borel's functor, basic questions in abstract harmonic analysis can be transfered from the group setting to the setting of Iwahori-Hecke algebras and graded Hecke algebras. For example, Barbasch and Moy \cite{BM} showed that the classification of unitary representations with Iwahori fixed vectors of a split reductive $p$-adic group can be reduced to the setting of $\C H$ and $\bH$. 

Motivated by the study of unitary representations, a Dirac operator and the notion of Dirac cohomology for $\bH$-modules were introduced and studied in \cite{BCT}. These constructions are analogues of the Dirac operator and Dirac cohomology for representations of real reductive groups. One of the main results concerning Dirac cohomology for $(\fg,K)$-modules, conjectured by Vogan and proved by Huang-Pand\v zi\'c \cite{HP}, says that, if nonzero, the Dirac cohomology of a $(\fg,K)$-module uniquely determines the infinitesimal character of the module. A graded Hecke algebra analogue is proved in \cite[Theorems 4.2 and 4.4]{BCT}. This Hecke algebra result will also be referred to as ``Vogan's conjecture'' in the sequel.   

\smallskip

In this paper we define the Dirac induction map $\Ind_D$ for $\mathbb{H}$. 
This map could be thought of as an algebraic analog of (the discrete 
part of) the explicit realization of the Baum-Connes assembly map 
$\mu: K^G_0(X)\to K_0(C_{\red}^*(G))$ (\cite{BC,BCH}) for the $K$-theory of the reduced 
$C^*$-algebra $C_{\red}^*(G)$ of a connected real reductive group $G$ using Dirac operators
(\cite{AS},\cite{P}; see \cite[Section 2.1]{La} for a concise account). 
In that context $X=G/K$ is the associated Riemannian symmetric space (which we assume to 
have a $G$-invariant spin structure) 
and $K^G_0(X)$ denotes the equivariant $K$-homology with $G$-compact supports.

Let $W$ be a real reflection group with reflection representation $V$. 
Let $k$ be a real valued $W$-invariant 
function on the set of simple reflections and let $\mathbb{H}$ be the 
associated graded affine Hecke algebra with parameters specialized at $k$ (Definition
\ref{d:graded}). 
In the present paper we describe the map $\Ind_D$ only on the discrete 
part of the equivariant $K$-homology with $G$-compact support of the group 
$G=V\rtimes W$ acting on $X=V$. After tensoring by $\mathbb{C}$ this space 
$K^G_{0,\disc}(X)$
can be identified with the complexified space $\overline{R}_{\mathbb{C}}(W)$ of virtual 
elliptic characters of $W$ (in the sense of Reeder \cite{R}, relative to the action of $W$ 
on $V$).  
This space comes equipped with a natural Hermitian inner product, the elliptic 
pairing. The map $\Ind_D$ which we construct is an isometric isomorphism 
$$\Ind_D:\overline{R}_{\mathbb{C}}(W)\to \overline{R}_{\mathbb{C}}(\mathbb{H})$$ 
where $\overline{R}_{\mathbb{C}}(\mathbb{H})$ is the complexified space of 
virtual elliptic characters of $\mathbb{H}$, equipped with its natural 
Euler-Poincar\'e pairing. 
The latter space can be identified canonically with the complexification of the discrete part 
of the Grothendieck group $K_0(\mathbb{H})$ of finitely generated projective modules 
over $\mathbb{H}$, elucidating the analogy with the (discrete part of) the assembly map 
as realized by Dirac operators.
We study the integrality properties of this map, and the central characters 
of $\Ind_D(\delta)$. An important role in this study is played by the Vogan conjecture 
as proved in \cite{BCT}, and the study of certain irreducible characters of the 
pin cover of $W$.

Next we consider a related analytically defined global elliptic Dirac operator between unitary 
representations of the graded affine Hecke algebra which are realized on the spaces of sections of 
certain vector bundles associated to a representation of the pin cover of the Weyl group.  
In this way we realize all irreducible discrete series modules of the Hecke algebra in the kernels (and indices) of such analytic Dirac operators. This can be viewed as a graded Hecke algebra analogue of the Parthasarathy and Atiyah-Schmid construction of discrete series representations for real semisimple Lie groups.

The results in this paper provide  direct links between three directions of 
research in the area of affine Hecke algebras:
\begin{enumerate}
\item the theory of the Dirac operator and Dirac cohomology, as defined for graded affine Hecke algebras in \cite{BCT,C,CT}; 
\item the Euler-Poincar\'e pairing and elliptic pairing of affine Hecke algebras and Weyl groups, as developed in \cite{OS1,R,S};
\item the harmonic analysis approach  of \cite{EOS,HO,O} to the study of unitary modules of graded Hecke algebras, particularly discrete series modules.
\end{enumerate}

We have used results on the existence and the central support of the discrete series modules of 
$\mathbb{H}$ from \cite{O}, \cite{O2}, \cite{OS1} in order to define the index of the 
global elliptic Dirac operators.
We hope to replace this by analytic results in the index theory of equivariant elliptic operators in 
a sequel to the present paper. This could hopefully also shed more light on formal degrees 
of discrete series representations of $\mathbb{H}$.

\subsection{}\label{s:1.2} Let us explain the main results of the
paper in more detail. Let $C(V)$ be the Clifford algebra defined with respect to $V$ and a $W$-invariant inner product $\langle~,~\rangle$ on $V$. Let $\wti W$ be the pin double cover of $W$, a subgroup of the group $\Pin(V).$ When $\dim V$ is odd, $C(V)$ has two nonisomorphic complex simple modules $S^+,S^-$ which remain irreducible when restricted to $\wti W$. When $\dim V$ is even, there is a single spin module $S$ of $C(V),$ whose restriction to the even part of $C(V)$ is a sum of two nonisomorphic simple modules $S^+,S^-.$ The modules $S^\pm$ are irreducible $\wti W'$-representations, where $\wti W'$ is the index two subgroup of $\wti W$ given by the kernel of the sign representation. In order to describe the results uniformly, denote $\wti W'=\wti W$, when $\dim V$ is odd. Let $W'$ be the image of $\wti W'$ in $W$ under projection.

In \cite{BCT}, an analogue of the classical Dirac element is defined: in our case, this is $\C D$, an element of $\bH\otimes C(V).$ For every finite-dimensional $\bH$-module $X$, left multiplication by $\C D$ gives rise to Dirac operators, which are $\wti W'$-invariant:
\begin{equation}
D^\pm:X\otimes S^\pm\to X\otimes S^\mp.
\end{equation}
The Dirac index of $X$ is the virtual $\wti W'$-module $I(X)=H_D^+-H_D^-$, where $H_D^\pm=\ker D^\pm/\ker D^\pm\cap\im D^\mp$ are the Dirac cohomology groups. A standard fact is that $I(X)=X\otimes (S^+-S^-)$, see Lemma \ref{l:CT}.

Let $R_\bC(\bH)$ be the complex Grothendieck group of $\bH$-modules, and let $\langle~,~\rangle^\EP_\bH$ denote the Euler-Poincar\'e pairing on $R_\bC(\bH)$ (section \ref{s:EP}). The radical of this form is spanned by parabolically induced modules, and let $\overline R_\bC(\bH)$ be the quotient by the radical, the space of (virtual) elliptic $\bH$-modules. 

The algebra $\bH$ contains a copy of the group algebra $\bC[W]$ of the
Weyl group as a subalgebra. The Grothendieck group $R_\bC(W)$ has an
elliptic pairing $\langle~,~\rangle^\el_W$ defined in \cite{R} whose
radical is spanned by induced representations (see section
\ref{s:Well}). Let $\overline R_\bC(W)$ be the quotient by the radical
of the form, the space of (virtual) elliptic $W$-representations. An
easy calculation noticed first in \cite{CT}, see Theorem \ref{t:CT},
shows that for every $\delta\in\overline R_\bZ(W)$, there exist
associate $\wti W'$-representations $\wti\delta^+$ and $\wti \delta^-$
such that $\langle \wti\delta^+,\wti\delta^-\rangle_{\wti W'}=0$ and 
\begin{equation}\label{e:delta+}
\delta\otimes (S^+-S^-)=\wti \delta^+-\wti\delta^-,\text{ and }\langle\delta,\delta\rangle^\el_W=\langle\wti\delta^+,\wti\delta^+\rangle_{\wti W'}.
\end{equation}

The relation between the elliptic theories of $\bH$ and $W$ is given by the restriction map. Precisely, combining results of \cite{OS,OS1,S} for the affine Hecke algebra and Lusztig's reduction theorems \cite{L}, one sees that the map
\begin{equation}
\rr: \overline R_\bC(\bH)\to \overline R_\bC(W),\quad [X]\mapsto [X|_W]
\end{equation}
is a linear isometry with respect to the pairing
$\langle~,~\rangle^\EP_\bH$ and $\langle~,~\rangle^\el_W.$ Our first
main result is the construction of an inverse $\Ind_D$ for the
isometry $\rr$, which we call the algebraic Dirac induction map. The
definition of $\Ind_D$ combines elements of K-theory with the Dirac
index of $\bH$-modules. Let $K_0(\bH)_\bC$ be the complex Grothendieck group of finitely generated projective $\bH$-modules and let $F_0H_0(\bH)$ be the subspace of projective modules with zero-dimensional support (Definition \ref{d:support}). The rank pairing (\ref{e:rkpair}) induces an injective linear map 
\begin{equation}
\Phi: F_0H_0(\bH)\to \overline R_\bC(\bH)
\end{equation}
with good properties with respect to the rank pairing and the Euler-Poincar\'e pairing (Lemma \ref{l:F0H0}). Then the map $\Ind_D$ is defined by extending linearly
\begin{equation}
\Ind_D(\delta)=\Phi([\bH\otimes_{W'}((\delta^+)^*\otimes (S^+-S^-))]),\quad \delta\in \overline R_\bZ(W).
\end{equation}
The following result is part of Theorem \ref{t:ind}.
\begin{theorem}
The map $\Ind_D$ is well-defined, and it is the inverse of the map $\rr$. Moreover,
\begin{equation}
\langle\Ind_D(\delta),X\rangle^\EP_\bH=\langle\delta,\rr(X)\rangle^\el_W,
\end{equation}
for every $\delta\in\overline R_\bC(W)$ and $X\in\overline R_\bC(\bH).$
\end{theorem}
As a consequence, we find that whenever $\delta$ is a rational
multiple of a pure element in $\overline R_\bZ(W)$, see Definition
\ref{d:pure}, then $\Ind_D(\delta)$ is supported in a single central
character $\Lambda(\delta)$, see Corollaries \ref{c:indmap} and
\ref{c:puresupport}. We also prove in Theorem \ref{t:ccmap}, that the
central character $\Ind_D(\delta)$ depends linearly in the parameter
function $k$ of the Hecke algebra. 

In section \ref{s:ortho}, we study further the map $\Ind_D$ in the
case when the root system $R$ is irreducible, in particular, we
investigate  its behaviour with respect to the integral lattices
$\overline R_\bZ(W)$ and $\overline R_\bZ(\bH)$, see section
\ref{s:integral}. To this end, in Theorem \ref{t:Honbasis} we find appropriate orthogonal bases
consisting of pure elements for the lattices $\overline R_\bZ(\bH)$
and $\overline R_\bZ(W).$

\subsection{}\label{s:1.3}
The second part of the paper concerns a realization of discrete series $\bH$-modules in the index (hence kernel) of certain global Dirac operators. This can be regarded as a Hecke algebra analogue of the construction in \cite{AS}. One complication in our setting is that, while the discrete series and tempered spectra of the graded affine Hecke algebra $\bH$ are known \cite{L2,O}, there is no known abstract Plancherel formula for $\bH.$ To bypass this difficulty we construct directly certain analytic models $\C X'_\omega(E\otimes S^\pm)$, which are (pre)unitary left $\bH$-modules, for every $\wti W'$-representation $E$. It is easy to check that every irreducible $\bH$-submodule of $\C X'_\omega(E\otimes S^\pm)$ is necessarily a discrete series module (Lemma \ref{l:ds1}). 

These models are analogous with the spaces of sections of spinor bundles for Riemannian symmetric spaces from \cite{P,AS}. The construction that we use is an adaptation of the ones from \cite{HO} and \cite{EOS} (the latter being in the setting of the trigonometric Cherednik algebra). A new ingredient  is an $\bH$-invariant inner product, this is Theorem \ref{t:*inv}, a generalization of the unitary structure from \cite{HO}. Once these definitions are in place, we can consider global Dirac operators acting ``on the right'':
\begin{equation}
D_E^\pm:\C X'_\omega(E\otimes S^\pm)\to \C X'_\omega(E\otimes S^\mp).
\end{equation}
If $\lambda$ is a central character of $\bH$, restrict $D_E^\pm$ to
the subspace $\C X'_\omega(E\otimes S^\pm)_\lambda$ of $\C
X'_\omega(E\otimes S^\pm)$ on which the center $Z(\bH)$ acts via
$\lambda.$ Denote by $D_E^\pm(\lambda)$ the restricted Dirac
operators.
Define the global Dirac index to be the formal expression
\begin{equation}
I_E=\bigoplus_\lambda ~ (\ker D_E^+(\lambda)-\ker D_E^-(\lambda)),
\end{equation}
a virtual left $\bH$-module. We show in section \ref{s:ds} that this
sum is finite. 

We define an analytic Dirac induction map $\Ind_D^\omega:\overline R_\bC(W)\to \overline R_\bC(\bH)$ as follows. Let $\delta\in \overline R_\bZ(W)$ be given, and let $\delta^+$ be the $\wti W'$-representation in (\ref{e:delta+}). Set
\begin{equation}
\Ind_D^\omega(\delta)=I_{(\wti\delta^+)^*}.
\end{equation}
A consequence of Vogan's conjecture (Theorem \ref{t:vogan}) is that
the irreducible $\bH$-modules that occur in $\Ind_D^\omega(\delta)$
must have a prescribed central character $\Lambda(\delta)$, see
Theorem \ref{t:ccmap}. The main results about $\Ind_D^\omega$ can be
summarized as follows (see Theorem \ref{t:globalindex}).
\begin{theorem}
\begin{enumerate}
\item For every irreducible $\bH$-module $X$, $\Hom_\bH(X,\Ind_D^\omega(\delta))=0$ unless $X$ is a discrete series module with central character $\Lambda(\delta)$, in which case
\begin{equation}
\dim\Hom_\bH(X,\Ind_D^\omega(\delta))=\langle \rr(X),\delta\rangle^\el_W.
\end{equation} 
\item Let $X$ be an irreducible discrete series $\bH$-module. Then  $X\cong \Ind_D^\omega(\rr(X))$  as $\bH$-modules. In particular, the image of the map $\Ind_D^\omega$ is the span of the discrete spectrum of $\bH$.
\end{enumerate}
\end{theorem}

\bigskip

\begin{notation}If $k$ is a unital commutative ring and $A$ is a $k$-associative algebra, denote by $\Rep(A)$ the category of (left) $A$-modules and by $\Rep_\fd(A)$ the category of finite-dimensional modules; if $G$ is a group, denote by $\Rep_k(G)$, the category of $k[G]$-modules. Let $R_K(A)$ and $R_K(G)$ denote the Grothendieck rings of $\Rep(A)$ and $\Rep_k(G)$, respectively, with coefficients in a ring $K$.
\end{notation}

\section{Preliminaries: Euler-Poincar\'e pairings}

\subsection{The root system}\label{s:roots} Fix a semisimple real root system
$\Phi=(V,R,V^\vee,R^\vee)$. In particular, $V$ and $V^\vee$ are finite-dimensional real vector spaces, $R\subset V\setminus\{0\}$ generates
$V$, $R^\vee\subset V^\vee\setminus\{0\}$ generates $V^\vee$ and there
is a perfect bilinear pairing 
$$(~,~): V\times V^\vee\to \bR,$$
which induces a bijection between $R$ and $R^\vee$ such that
$(\al,\al^\vee)=2$ for all $\al\in R.$ 
For every $\al\in R$, let $s_\al: V\to V$ denote the reflection about
the root $\al$ given by $s_\al(v)=v-(v,\al^\vee)\al$ for all $v\in V.$
We also identify $s_\al$ with the map $s_\al: V^\vee\to V^\vee,$
$s_\al(v')=v'-(\al,v')\al^\vee$. 

Let $W$ be the subgroup of $GL(V)$
generated by $\{s_\al:\al\in R\}$; we may also regard $W$ as a
subgroup of $GL(V^\vee).$ Fix a basis $F$ of $R$ and set $S=\{s_\al:
\al\in F\}.$ Then $(W,S)$ is a finite Coxeter group.

Let $Q=\bZ R\subset V$ denote the root lattice, $\C P\subset V$ the weight lattice,  and form the affine Weyl group
$W_\af=Q\rtimes W.$ Let $F_\af$ be the set of simple
affine roots, and $S_\af\supset S$ the corresponding set of simple affine
reflections. Let $\ell$ denote the length function of the Coxeter
group $(W_\af,S_\af).$

The complexified vector spaces will be denoted by $V_\bC$ and $V^\vee_\bC$.

\subsection{The affine Hecke algebra $\C H$}

Let $q:S_\af\to \bR_{>0}$ be a function such that $q_s=q_{s'}$
whenever $s,s'\in S_\af$ are $W_\af$-conjugate. For every $s\in
S_\af,$ let $q_s^{1/2}$ denote the positive square root of $q_s.$

\begin{definition}
The affine Hecke algebra $\CH=\CH(W_\af,q)$ is the unique associative
unital free $\bC$-algebra with basis $\{T_w: w\in W_\af\}$ subject to
the relations:
\begin{enumerate}
\item[(i)] $T_w T_{w'}=T_{w w'}$ if $\ell(w w')=\ell(w)+\ell(w')$;
\item[(ii)] $(T_s-q_s^{1/2})(T_s+q_s^{-1/2})=0,$ for all $s\in S_\af.$
\end{enumerate}
\end{definition}
A particular instance is when the parameters $q_s$ of $\CH$ are
specialized to $1$; then $\CH$ becomes $\bC[W_\af]$.

A result of Bernstein (\cf \cite[Proposition 3.11]{L}) says that the
center $Z(\CH)$ of $\C H$ is isomorphic with the algebra of $W$-invariant
complex functions on $Q$. In particular, the central characters, i.e.,
the homomorphisms $\chi^t: Z(\CH)\to \bC$ are parameterized by classes
$W t$ in $W\backslash T$, where $T$ is the complex torus
$\bC\otimes_\bZ \C P.$ We say that
the central character $\chi^t$ is real if $t\in \bR\otimes_\bZ \C P.$

Denote by $\Rep(\CH)_0$ and  $\Rep_\fd(\CH)_0$  the respective full subcategories of
$\CH$-modules with real central characters.
If $P\subset F$, let $W_P$ denote the subgroup of $W$ generated by
$\{s_\al:\al\in P\}.$ Let $\CH_P$ be the parabolic affine Hecke
algebra, see \cite[section 1.4]{S} for example.

Delorme-Opdam defined the Schwartz algebra $\C S$ of $\CH$ (\cite[section 2.8]{DO}) by
\begin{equation}
\C S=\{\sum_{w\in W_\af}c_w T_w: \text{ for all }n\in \mathbb N,
(1+\ell(w))^nc_w \text{ is bounded as a function in } w\in W_\af\}.
\end{equation}
Then $\C S$ is a nuclear Fr\'echet algebra (in the sense of \cite[Definition 6.6]{O}) with respect to the family
of norms 
$$p_n(\sum_{w\in W_\af}c_wT_w):=\sup\{(1+\ell(w))^n|c_w|: w\in W_\af\}.$$ 
An irreducible module $X\in \Rep_\bC(\CH)$ is called tempered if it
can be extended to an $\C S$-module. The irreducible summands in the
$\CH$-decomposition of $\C S$ are called discrete series modules. 
These definitions of tempered and discrete series modules agree with those
given by the Casselman criterion which we do not recall here, instead
we refer to \cite[section 2.7]{O} for the details. Let $\Rep(\C S)$ denote the
category of $\C S$-representations.

\subsection{The graded affine Hecke algebra $\bH$}

Let $\underline k=\{\underline k_\al:\al\in F\}$ be a set of indeterminates such that $\underline k_\al=\underline k_{\al'}$ whenever $\al,\al'$ are $W$-conjugate. Denote $A=\bC[\underline k]$.
Let $\bC[W]$ denote the group
algebra of $W$ and $S(V_\bC)$ the symmetric algebra over $V_\bC.$ The
group $W$ acts on $S(V_\bC)$ by extending the action on $V.$ For every
$\al\in F,$  denote the difference operator by
\begin{equation}\label{e:diffop}
\Delta: S(V_\bC)\to S(V_\bC),\quad
\Delta_\al(p)=\frac{p-s_\al(p)}{\al},\text{ for all }p\in S(V_\bC).
\end{equation}

\begin{definition}[\cite{L}]\label{d:graded}
The generic graded affine Hecke algebra $\bH_A=\bH(\Phi,F,\underline k)$  is the unique
associative unital $A$-algebra such that 
\begin{enumerate}
\item[(i)] $\bH_A=S(V_\bC)\otimes A[W]$ as a $(S(V_\bC),A[W])$-bimodule;
\item[(ii)] $s_\al\cdot p=s_\al(p)\cdot s_\al+\underline k_\al \Delta_\al(p),$
  for all $\al\in F$, $p\in S(V_\bC).$
\end{enumerate}

If $k: F\to \bR_{\ge 0}$ is a function such that $k_\al=k_{\al'}$ whenever
$\al,\al'\in F$ are $W$-conjugate, let $\bH$ (or $\bH_k$, when we wish
to emphasize the dependence on the parameter function $k$) be the
specialization of $\bH_A$ at $k,$ i.e., $\bH=\bC_k\otimes_A \bH_A$,
where $\bC_k$ is the $A$-module on which $\underline k$ acts by $k$.
\end{definition}

A result of Lusztig \cite[Proposition 4.5]{L} says that the center
$Z(\bH_A)$ of $\bH_A$ is $A\otimes_\bC S(V_\bC)^W.$ In particular, $Z(\bH)=S(V_\bC)^W$ and the central
characters $\chi^\lambda: Z(\bH)\to \bC$ are parameterized by classes
$W\lambda$ in $W\backslash V^\vee_\bC$. We say that the central character
$\chi^\lambda$ is real if $\lambda\in V^\vee.$

Denote by $\Rep(\bH)_0$ and $\Rep_\fd(\bH)_0$ the respective full subcategories of
$\bH$-modules with real central characters.
If $P\subset F$, let $\bH_P$ be the parabolic graded affine Hecke subalgebra, see \cite[section 1.4]{S} for example.

\begin{definition}An $\bH$-module $X$ is called tempered if every $S(V_\bC)$-weight
$\nu\in V_\bC^\vee$ of $X$ satisfies the Casselman criterion
\begin{equation}\label{bHtemp}
(\omega,\Re\nu)\le 0,\text{ for all fundamental weights }\omega\in \C P.
\end{equation}
The module $X$ is called a discrete series module if all the inequalities in
(\ref{bHtemp}) are strict. Denote by $\ds(\bH)$ the set of irreducible discrete series $\bH$-modules.
\end{definition}

A particular case of Lusztig's reduction theorems is the following.

\begin{theorem}
There is an equivalence of categories $$\eta:\Rep_\fd(\CH(W_\af,q))_0\to
\Rep_\fd(\bH_k)_0,$$ where the relation between the parameters $q$ and
$k$ is as in \cite[(26)]{OS1}. Moreover, $X\in \Rep_\fd(\CH(W_\af,q))_0$ is irreducible tempered
(resp. discrete series) module if and only if $\eta(X)\in \Rep_\fd(\bH_k)_0$ is irreducible
tempered (resp. discrete series) module.
\end{theorem}

\subsection{Euler-Poincar\' e pairing for the Weyl group}\label{s:Well}
If
$\sigma\in R_\bC(W),$ let $\chi_\sigma$ denote the character of $\sigma.$

\begin{definition}[{\cite[section 2.1]{R}}] For $\sigma,\mu\in R_\bC(W)$, the
  Euler-Poincar\'e pairing is the Hermitian form
$$\langle\sigma,\mu\rangle_W^\el=\frac 1{|W|}\sum_{w\in
    W}\overline{\chi_\sigma(w)}\chi_\mu(w) \Det_{V_\bC}(1-w).$$
An element $w\in W$ is called elliptic if $\Det_{V_\bC}(1-w)\neq 0.$ 
\end{definition}

For every parabolic subgroup $W_{P}$ of $W$ corresponding to 
$P\subset F$, let $\ind_{W_{P}}^W: \Rep_\bC(W_{P})\to \Rep_\bC(W)$
denote the induction functor. 

\begin{theorem}[{\cite[(2.1.1), (2.2.2)]{R}}] The radical of the form
  $\langle~,~\rangle_W^\el$ on $R_\bC(W)$ is $\sum_{P\subsetneq
    F}\ind_{W_{P}}^W(R_\bC(W_{P}))$. Moreover:  
\begin{enumerate}
\item[(i)]  $\langle~,~\rangle_W^\el$ induces a positive definite Hermitian
  form on $$\overline R_\bC(W)=R_\bC(W)/\sum_{P\subsetneq
    F}\ind_{W_{P}}^W(R_\bC(W_{P})).$$
\item[(ii)] the dimension of $\overline R_\bC(W)$ equals the number
  of elliptic conjugacy classes in $W$.
\end{enumerate}
\end{theorem}

\subsection{Euler-Poincar\' e pairing for $\C H$}\label{s:EP}

\begin{definition}[{\cite[(3.15), Theorem 3.5]{OS}}]
Define on $R_\bC(\CH)$ a Hermitian form
$\langle~,~\rangle^{\EP}_\CH$ by 
$$\langle X,Y\rangle^\EP_\CH=\sum_{i\ge
  0}(-1)^i\dim\Ext_\CH^i(X,Y),\text{ for all }X,Y\in R_\bC(\CH).$$
\end{definition}

Let $\ind_{\CH_P}^\CH:\Rep_\bC(\CH_P)\to \Rep_\bC(\CH)$ denote the
induction functor.
If $P\subsetneq F,$ one can see that the space $\ind_{\CH_P}^\CH(R_\bC(\CH_P))$ is in
the radical of $\langle~,\rangle^\EP_\CH$, \cite[Proposition 3.4]{OS} Therefore,
$\langle~,\rangle^\EP_\CH$ factors through 
$$\overline R_\bC(\CH)=R_\bC(\CH)/\sum_{P\subsetneq
    F}\ind_{\CH_{P}}^\CH(R_\bC(\CH_{P})).$$

One can define the pairing $\langle~,~\rangle^{\EP}_\C S$ in
$R_\bC(\C S)$ and consider the space of (virtual) elliptic tempered
modules $\overline R_\bC(\C S).$ As a consequence of the Langlands
classification, one sees easily that
$\overline R_\bC(\CH)=\overline R_\bC(\C S)$. The fact that this isomorphism is an isometry with respect to the Euler-Poincar\'e pairings follows from \cite[Corollary 3.7]{OS}:
\begin{equation}
\Ext^i_\CH(X,Y)\cong \Ext^i_\CS(X,Y),\text{ for all finite-dimensional tempered $\CH$-modules } X,Y.
\end{equation}

\begin{theorem}[{\cite[Theorem 3.8]{OS}}] If $X$ is an irreducible discrete series $\CH$-module and $Y$ is a finite-dimensional tempered $\CH$-module, then
\begin{equation}
\Ext^i_\CH(X,Y)=\Ext^i_\CS(X,Y)=0,\text{ for all }i>0.
\end{equation}
In particular, $\langle X,Y\rangle^\EP_\CH=1$ if $X\cong Y$ and $0$ otherwise.
\end{theorem}

Solleveld \cite[Theorem 4.4.2]{S}
defines a scaling map
\begin{equation}
\tau_0: R_\bC(\C S)\to R_\bC(W_\af),\quad
\tau_0(X)=\lim_{\ep\to 0} X|_{q_s\to q_s^\ep};
\end{equation}
this has the property that it factors through $\overline R_\bC(\C
S)=\overline R_\bC(\CH)\to \overline R_\bC(W_\af).$ The relation with
the elliptic theory of the Weyl group can be summarized in the
following statements which are a combination of \cite[Theorem 3.2,
Proposition 3.9]{OS} together with 
\cite[Corollary 7.4]{OS1} and \cite[Theorem 2.3.1]{S}.

\begin{theorem}[Opdam-Solleveld]
\begin{enumerate}
\item[(i)] The map $\tau_0$ defines a linear isomorphism $\overline R_\bC(\CH)\to \overline R_\bC(W_\af)$, $[X]\mapsto [\tau_0(X)].$
\item[(ii)] The isomorphism $\tau_0$ from (i) is an isometry with respect to $\langle~,~\rangle^\EP_\CH$ and $\langle~,~\rangle^\EP_{W_\af}$, respectively.
\item[(iii)] The map $\tau_{\el}: \overline R_\bC(\CH)_0\to \overline R_\bC(W),$ given by $[X]\mapsto [\tau_0(X)|_W]$ is a linear isomorphism and an isometry with respect to $\langle~,~\rangle^\EP_\CH$ and $\langle~,~\rangle^\el_W$.
\end{enumerate}
\end{theorem}

\subsection{Euler-Poincar\' e pairing for $\bH$}\label{s:EPH}
We need to translate the previous results into the setting of the graded Hecke algebra $\bH$. Define the pairing $\langle~,~\rangle^\EP_\bH$ on $R_\bC(\bH)$ and the space of virtual elliptic modules $\overline R_\bC(\bH)$ in the same way as for $\CH.$ Lusztig's reduction equivalence $\eta:\Rep_\fd(\CH)_0\to \Rep_\fd(\bH)_0$ induces a linear isomorphism $\overline\eta: \overline R_\bC(\CH)_0\to\overline R_\bC(\bH)_0$ which is an isometry with respect to the Euler-Poincar\'e pairings. In fact, $\overline R_\bC(\bH)_0=\overline R_\bC(\bH),$ i.e., all finite-dimensional elliptic $\bH$-modules have real central character. Comparing with the results recalled before, we arrive at the following corollaries.

\begin{corollary}\label{c:EPds}
If $X$ is an irreducible discrete series $\bH$-module, and $Y$ is a finite-dimensional tempered $\bH$-module, then $\langle X,Y\rangle^\EP_\bH=1$ if $X\cong Y$ and $0$ otherwise.
\end{corollary}

\begin{corollary}\label{c:resmap}
The restriction map $\res:\Rep_\bC(\bH)\to \Rep_\bC(W)$, $\res(X)=X|_W$ induces a linear isometry 
\begin{equation}\label{e:rmap}
\rr: \overline R_\bC(\bH)\to \overline R_\bC(W)
\end{equation}
with respect to $\langle~,~\rangle^\EP_\bH$ and $\langle~,~\rangle^\el_W,$ respectively. In particular, if $X$ is an irreducible  discrete series $\bH$-module, and $Y$ is a finite-dimensional tempered $\bH$-module, $\langle \rr(X), \rr(Y)\rangle^\el_W=1$ if $X\cong Y$ and $0$ otherwise.
\end{corollary}

\begin{comment}
Define $\ET(\bH)$ to be the set of irreducible tempered $\bH$-modules $X$ such that $X\neq 0$ in $\overline R_\bZ(\bH).$
\end{comment}

\section{Preliminaries: Dirac operators for the graded affine Hecke algebra}

We retain the notation from the previous section. We recall the construction and basic facts about the Dirac operator for $\bH$.

\subsection{The pin cover of the Weyl group} Fix a $W$-invariant inner product $\langle~,~\rangle$ on $V$ and let $C(V)$ denote the Clifford algebra, the quotient of the tensor algebra of $V$ by the ideal generated by $\{\om\otimes\om'+\om'\otimes\om+2\langle\om,\om'\rangle:\om,\om'\in V\}.$ Let $O(V)$ denote the group of orthogonal transformations of $V$ with respect to $\langle~,~\rangle.$ We have $W\subset O(V).$ The action of $-1\in O(V)$ on $C(V)$ induces a $\bZ/2\bZ$-grading $C(V)=C(V)_{\mathsf{even}}+C(V)_{\mathsf{odd}}$, and let $\ep$ be the automorphism of $C(V)$ which is $1$ on $C(V)_{\mathsf{even}}$ and $-1$ on $C(V)_{\mathsf{odd}}.$ Let $^t$ be the transpose automorphism of $C(V)$ defined by $\om^t=-\om$, $\om\in V$, and $(ab)^t=b^ta^t$ for $a,b\in C(V).$ The pin group is
\begin{equation}
\Pin(V)=\{a\in C(V)^\times: \ep(a)Va^{-1}\subset V,\ a^t=a^{-1}\};
\end{equation}
it is a central $\bZ/2\bZ$-extension of $O(V)$:
$$1\to\{\pm 1\}\to\Pin(V)\overset{p}\to O(V)\to1,$$
where $p$ is the projection $p(a)(\om)=\ep(a)\om a^{-1}.$ Construct the central $\bZ/2\bZ$-extension $\wti W=p^{-1}(W)$ of $W$:
\begin{equation}
1\to\{\pm 1\}\to\wti W\overset{p}\to W\to1.
\end{equation}
The group $\wti W$ has a Coxeter presentation similar to that of $W$, see \cite{M}:
\begin{equation}\label{e:coxeterWtil}
\wti W=\langle z, \wti s_\al, \al\in F: z^2=1,\ (\wti s_\al\wti s_\beta)^{m(\al,\beta)}=z, ~\al,\beta\in F\rangle.
\end{equation}
With this presentation, the embedding of $\wti W$ in $\Pin(V)$ is given by:
\begin{equation}\label{e:embedWtil}
z\mapsto -1,\quad \wti s_\al\mapsto \frac 1{|\al|}\al.
\end{equation}

\subsection{The Dirac element}\label{s:3.2}
The generic Hecke algebra $\bH_A$ (Definition \ref{d:graded}) has a natural $*$-operation coming from the relation with the affine Hecke algebra $\C H$ and $p$-adic groups. On generators, this is defined by
\begin{equation}
\begin{aligned}
&\underline k_\al^*=\underline k_\al,\ \al\in F;\quad w^*=w^{-1},\ w\in W;\\
&\xi^*=-w_0\cdot w_0(\xi)\cdot w_0=-\xi+\sum_{\beta\in R^+}\underline k_\beta (\xi,\beta^\vee) s_\beta,\ \xi\in V.
\end{aligned}
\end{equation}

For every $\xi\in V$, define
\begin{equation}\label{e:Txi}
\wti\xi=\xi-T_\xi,\text{ where }T_\xi=\frac 12\sum_{\beta\in R^+}\underline k_\beta (\xi,\beta^\vee) s_\beta\in \bH_A.
\end{equation}
Then $\wti\xi^*=-\wti\xi,$ for all $\xi\in V.$

\begin{definition}\label{d:diracelem}
Let $\{\xi_i\},\{\xi^i\}$ be dual bases of $V$ with respect to $\langle~,~\rangle$. The Dirac element is
$$\C D=\sum_i\wti\xi_i\otimes\xi^i\in\bH_A\otimes C(V).$$
It does not depend on the choice of bases. 
\end{definition}
Write $\rho$ for the diagonal embedding of $\bC[\wti W]$ into $\bH_A\otimes C(V)$ defined by extending linearly $\rho(\wti w)=p(\wti w)\otimes \wti w.$ By \cite[Lemma 3.4]{BCT}, we see that $\C D$ is $\sgn$ $\wti W$-invariant, i.e.,
\begin{equation}\label{e:sgninv}
\rho(\wti w)\C D=\sgn(\wti w) \C D \rho(\wti w),\text{ for all }\wti w\in\wti W.
\end{equation}
Moreover, \cite[Theorem 3.5]{BCT} computes $\C D^2$, and in particular shows that $\C D^2$ acts diagonally on $\wti W$-isotypic components of $X\otimes S$, for every irreducible $\bH$-module $X$ and every $C(V)$-module $S$. More precisely, define
\begin{equation}\label{e:omega}
\Omega=\sum_i\xi_i\xi^i\in Z(\bH_A),
\end{equation}
and
\begin{equation}\label{e:omegaWtil}
\Omega_{\wti W}=\frac z 4\sum_{\al>0,\beta>0,s_\al(\beta)<0}\underline k_\al \underline k_\beta |\al^\vee||\beta^\vee|\wti s_\al \wti s_\beta\in \bC[\wti W]^{\wti W}.
\end{equation}
Then, we have
\begin{equation}\label{e:Dsquare}
\C D^2=-\Omega\otimes 1+\rho(\Omega_{\wti W}).
\end{equation}

\subsection{Vogan's conjecture} 
In this section, we sharpen the statement of Vogan's conjecture (proved in \cite{BCT}) related to the Dirac cohomology of $\bH$-modules. The following result is the generic analogue of \cite[Theorem 4.2]{BCT}.

\begin{theorem}\label{t:vogan} Let $\bH_A$ be the generic graded Hecke algebra over $A=\bC[\underline k]$ (see Definition \ref{d:graded}). Let $z\in Z(\bH_A)$ be given. Then there exist $a\in\bH_A\otimes C(V)_\odd$ and a unique element $\zeta(z)$ in the center of $A[\wti W]$ such that 
\begin{equation}\label{e:defzeta}
z\otimes 1=\rho(\zeta(z))+\C D a+a\C D,
\end{equation}
as elements in $\bH_A\otimes C(V).$ Moreover, the map $z\to \zeta(z)$ defines an algebra homomorphism $\zeta: Z(\bH_A)\to A[\wti W]^{\wti W}.$
\end{theorem}
Notice that under the homomorphism $\zeta$ we have
\begin{equation}\label{e:zetaomega}
\zeta(\Omega)=\Omega_{\wti W},
\end{equation}
where $\Omega$ is as in (\ref{e:omega}) and $\Omega_{\wti W}$ is as in (\ref{e:omegaWtil}).

\begin{comment}
\begin{proposition}\label{p:vogan2}
Let $\zeta_k:Z(\bH_k)\to \bC[\wti W]^{\wti W}$ be the algebra homomorphism from Theorem \ref{t:vogan}. Then $\zeta_k(z)\in \bC[\wti W]^{\wti W}[k],$ for every $z\in Z(\bH_k)=S(V_\bC)^W$.
\end{proposition}
\end{comment}

\begin{proof} We proceed as in \cite[section 4]{BCT}. Let $\bH_A^0\subset \bH_A^1\subset\dots\bH_A^n\subset\dots$ be the filtration coming from the degree filtration of $S(V_\bC)$. (The group algebra $A[W]$ has degree zero.) The associated graded object $\oplus_j \overline \bH_A^j$, $\overline\bH_A^j=\bH_A^j/\bH_A^{j-1}$, is naturally isomorphic as an $A$-algebra to $\bH_{A,0}=A\otimes_\bC \bH_{0}.$ Define 
$$d:\bH_A\otimes C(V)\to \bH_A\otimes C(V),$$
by extending linearly $d(h\otimes c_1\dots c_\ell)=\C D\cdot (h\otimes c_1\dots c_\ell)-(-1)^\ell(h\otimes c_1\dots c_\ell)\cdot \C D$, $h\in\bH_A,$ $c_i\in V$, and restrict $d$ to 
$$d^\triv:(\bH_A\otimes C(V))^\triv\to (\bH_A\otimes C(V))^\sgn,$$
as in \cite[section 5]{BCT}. The statement of Theorem \ref{t:vogan} follows from
$$\ker(d^\triv)=\im(d^\sgn)\oplus\rho(A[\wti W]^{\wti W}),$$
see \cite[Theorem 5.1]{BCT}. This in turn is proved as follows. 
Firstly, one verifies that $\ker(d^\triv)\supset \im(d^\sgn)\oplus\rho(A[\wti W]^{\wti W})$.
Secondly, $d^\triv$ induces a graded differential
$$\overline d^\triv:(\bH_{A,0}\otimes C(V))^\triv\to (\bH_{A,0}\otimes C(V))^\sgn,$$
for which \cite[Corollary 5.9]{BCT} shows that
\begin{equation}\label{e:diffgraded}
\ker (\overline d^\triv)=\im (\overline d^\sgn)\oplus \overline \rho(A[\wti W]^{\wti W}).
\end{equation}
Thirdly, one proceeds by induction on  degree in $\bH_A\otimes C(V)$ to deduce from (\ref{e:diffgraded}) the opposite inclusion $\ker(d^\triv)\subset \im(d^\sgn)\oplus\rho(A[\wti W]^{\wti W})$. This is the step that we need to examine more closely. 

Let $b\in \ker(d^\triv)$ be an element of $\bH_A^n\otimes C(V)$ (i.e., an element of degree $n$ in the filtration), which can be specialized to $z\otimes 1$. Since $d^\triv(b)=0$, taking the graded objects, we have $\overline d^\triv(\overline b)=0$ in $\bH_{A,0}\otimes C(V).$ From (\ref{e:diffgraded}), there exists $\overline c\in  (\bH^{n-1}_{A,0}\otimes C(V))^\sgn$ and $s\in A[\wti W]^{\wti W}$ such that 
$$\overline b=\overline d^\sgn\overline c+\overline\rho(s).$$
Choose $c\in (\bH^{n-1}_A\otimes C(V))^\sgn$ such that $\overline c$ is the image of $c$ in $(\bH^{n-1}_{A,0}\otimes C(V))^\sgn$. For example, if $\overline c=\sum w\xi_{w,1}\dots\xi_{w,n-1}\otimes f_w$, $w\in W,$ $\xi_{w,i}\in V_\bC,$ $f_w\in C(V)$, we can choose $c=\sum w\wti \xi_{w,1}\dots\wti\xi_{w,n-1}\otimes f_w$. Then
$$\overline{b-d^\sgn c-\rho(s)}=\overline b-\overline d^\sgn\overline c-\overline\rho(s)=0,$$
hence $b-d^\sgn c-\rho(s)\in (\bH_A^{n-1}\otimes C(V))^\triv.$ On the other hand, 
$$d^\triv(b-d^\sgn c-\rho(s))=d^\triv(b)-d^2(c)-d^\triv(\rho(s))=0,$$
where $d^\triv(b)=0$ by assumption, $d^2(c)=0$ by \cite[Lemma 5.3]{BCT}, and $d^\triv(\rho(s))=0$ by \cite[Lemma 5.2]{BCT}. But then, by induction, $b-d^\sgn c-\rho(s)=d^\sgn  c'+\rho(s'),$ where $s'\in A[\wti W]^{\wti W}$ and $ c'\in (\bH_A\otimes C(V))^\sgn.$ 

\end{proof}

\begin{corollary}\label{c:imagezeta}
Let $\zeta:Z(\bH_A)\to A[\wti W]^{\wti W}$ be the algebra homomorphism from Theorem \ref{t:vogan}. Then the image of $\zeta$ lies in $A[\ker\sgn]$, where $\sgn$ is the sign $\wti W$-representation.
\end{corollary}

\begin{proof}
For every $z\in Z(\bH_A),$ consider the defining equation (\ref{e:defzeta}) for $\zeta(z).$ Since $a\in \bH_A\otimes C(V)_\odd$, we have $\C D a+a \C D\in \bH_A\otimes C(V)_\even$, and also clearly $z\otimes 1\in \bH_A\otimes C(V)_\even.$ Thus $\rho(\zeta(z))\in \bH_A\otimes C(V)_\even$, and the claim follows.
\end{proof}

\section{Dirac induction (algebraic version)}\label{s:4}
The goal of this section is to construct an inverse of the restriction map $\rr$ from (\ref{e:rmap}), using the Dirac index theory.

\subsection{The local Dirac index}\label{s:3.4} The notion of Dirac index for finite-dimensional $\bH$-modules was introduced in \cite[section 2.9]{CT}. Let $X$ be a finite-dimensional $\bH$-module. Denote
\begin{equation}
\wti W'=\begin{cases}
\wti W, & \text{ if $\dim(V)$ is odd,}\\
\ker\sgn, & \text{ if $\dim(V)$ is even.}
\end{cases}
\end{equation}
Set also $W'=p(\wti W')\subset W.$

Assume first that $\dim V$ is even. Then $C(V)$ has a unique complex simple module $S$ whose restriction to $C(V)_\even$ splits into the sum of two inequivalent complex simple $C(V)_\even$-modules $S^+,S^-.$ The Dirac operator  $D\in \End_{\bH\otimes C(V)}(X\otimes S)$ is the endomorphism given by the action of the Dirac element $\C D$. When restricted to $\bH\otimes C(V)_\even$,  $D$ maps $X\otimes S^\pm$ to $X\otimes S^\mp$, and denote by $D^\pm:X\otimes S^\pm\to X\otimes S^\mp$, the corresponding restrictions. From (\ref{e:sgninv}), we see that $D^\pm$ both commute with the action of $\wti W'$. Define the Dirac cohomology of $X$ to be the $\wti W$-representation 
\begin{equation}\label{e:diraccoh}
H_D(X)=\ker D/\ker D\cap\im D.
\end{equation}
We may also define cohomology with respect to $D^\pm$ as the $\wti W'$-representations
\begin{equation}\label{e:diraccohpm}
H^+_D(X)=\ker D^+/\ker D^+\cap\im D^-\text{ and }H_D^-=\ker D^-/\ker D^-\cap\im D^+.
\end{equation}
Then the index of $X$ is the virtual $\wti W'$-module
\begin{equation}\label{e:diracindex}
I(X)=H^+_ D- H^-_D.
\end{equation}
Now consider the case when $\dim V$ is odd. 
The Clifford algebra $C(V)$ has two nonisomorphic complex simple modules $S^+,S^-$. The restriction of $S^+$ and $S^-$ to $C(V)_\even$ are isomorphic, and they differ by the action of the center $Z(C(V))\cong \bZ/2\bZ.$ As before, we have the Dirac operator $D\in \End (X\otimes S^+)$ given by the action of $\C D$. By (\ref{e:sgninv}), $D^+$ is $\sgn$ $\wti W$-invariant, but by composing $D$ with the vector space isomorphism $S^+\to S^-$, we may regard $D^+:X\otimes S^+\to X\otimes S^-$ as $\wti W$-invariant. Similarly, we define $D^-$. Then the definitions (\ref{e:diraccoh}), (\ref{e:diraccohpm}) and (\ref{e:diracindex}) make sense in this case as well. Notice that the Dirac index of $I(X)$  is a virtual $\wti W'=\wti W$-module.

Recall that $S^\pm$ admit structures of unitary $C(V)_\even$-modules. If $X$ is unitary (or just Hermitian), let $(~,~)_{X\otimes S^\pm}$ denote the tensor product Hermitian form on $X\otimes S^\pm.$ Then $D^+,D^-$ are adjoint with respect to $(~,~)_{X\otimes S^\pm}$, i.e.,
\begin{equation}\label{Dadj}
(D^+x,y)_{X\otimes S^-}=(x,D^-y)_{X\otimes S^+}, \text{ for all }x\in X\otimes S^+,\ y\in X\otimes S^-.
\end{equation}

The following result is standard.

\begin{lemma}\label{l:CT}
As virtual $\wti W'$-modules, $I(X)=X\otimes S^+-X\otimes S^-,$ for every finite-dimensional $\bH$-module $X$. 
\end{lemma}

\begin{proof}
Let $\wti \sigma$ be an irreducible $\wti W'$-module.  Let $D^\pm_{\wti\sigma}$ be the restrictions of $D^\pm$ to the $\wti\sigma$-isotypic component $(X\otimes S^\pm)_{\wti\sigma}$ in $X\otimes S^\pm$, respectively. There are two cases:
\begin{enumerate}
\item[(i)] $D^+_{\wti\sigma}D^-_{\wti\sigma}\neq 0$ or equivalently $D^-_{\wti\sigma}D^+_{\wti\sigma}\neq 0$. In this case, $D^+_{\wti\sigma}$ and $D^-_{\wti\sigma}$ are both isomorphisms, and thus the identity in the Lemma is trivially verified.
\item[(ii)] $D^+_{\wti\sigma}D^-_{\wti\sigma}= 0$ or equivalently $D^-_{\wti\sigma}D^+_{\wti\sigma}=0$. Consider the complex
\begin{equation}
0\to \ker D^+_{\wti\sigma}\to (X\otimes S^+)_{\wti\sigma}\overset{D^+_{\wti\sigma}}\longrightarrow (X\otimes S^-)_{\wti\sigma}\overset{D^-_{\wti\sigma}}\longrightarrow \im D^-_{\wti\sigma}\to 0.
\end{equation}
\end{enumerate}
The claim follows from the Euler principle.

\end{proof}

Two important observations are that
 \begin{equation}\label{e:obs1}
(S^+-S^-)\otimes (S^+-S^-)^*\cong \frac 2{[\wti W:\wti W']} \sum_{i=0}^n (-1)^i \wedge^i V_\bC,\text { as virtual }\wti W'\text{-modules},
\end{equation}
and 
that the character of  $\sum_{i=0}^n (-1)^i \wedge^i V_\bC$ on $w\in W$ is 
\begin{equation}\label{e:obs2}
\Det_{V_\bC}(1-w).
\end{equation}

In particular, since $S^+-S^-$ is supported on the elliptic set,
\begin{equation}\label{e:vanishinduced}
I(\ind_{\bH_P}^\bH(R_\bC(\bH_P))=0,
\end{equation}
for every proper parabolic subalgebra $\bH_P$ of $\bH$. Let $R(\wti W')_\gen$ be the Grothendieck group of $\wti W'$ spanned by the genuine representations, and let $\langle~,~\rangle_{\wti W'}$ be the usual character pairing on $R(\wti W').$

We define an involution 
\begin{equation}\label{e:Sg}
\Sg: R_\bZ(\wti W')\to R_\bZ(\wti W'),
\end{equation}
as follows. If $\dim V$ is odd, $\wti W'=\wti W$, and set $\Sg(\sigma)=\sigma\otimes\sgn$ for every $\sigma\in R_\bZ(\wti W').$ Suppose that $\dim V$ is even. Assume $\wti\sigma\cong\wti\sigma\otimes\sgn,$ for an irreducible $\wti W$-representation $\wti\sigma.$ Then $\wti\sigma$ restricts to $\wti W'$ as a sum $\wti\sigma^+\oplus\wti\sigma^-$ of two inequivalent irreducible representation of the same dimension. Set $\Sg(\wti\sigma^\pm)=\wti\sigma^\mp.$ If $\wti\sigma\not\cong\wti\sigma\otimes\sgn,$ then $\wti\sigma|_{\wti W'}\cong \wti\sigma\otimes\sgn|_{\wti W'}$ and in this case set $\Sg(\wti\sigma|_{\wti W'})=\wti\sigma|_{\wti W'}.$ This assignment extends to an involution $\Sg: R_\bZ(\wti W')\to R_\bZ(\wti W')$.

For every $\wti\sigma_1,\wti\sigma_2\in R_\bZ(\wti W'),$ one has
\begin{equation}
\langle\wti\sigma_1,\wti\sigma_2\rangle_{\wti W'}=\langle\Sg(\wti\sigma_1),\Sg(\wti\sigma_2)\rangle_{\wti W'}.
\end{equation}
Then we obtain the following result.
\begin{theorem}[\cite{CT}]\label{t:CT}
\begin{enumerate}
\item The map $i: R_\bZ(W)\to R_\bZ(\wti W')_\gen$, $\delta\mapsto \delta\otimes(S^+-S^-)$, gives rise to an injective map $i:\overline R_\bC(W)\to R_\bC(\wti W')_\gen$ which satisfies
\begin{equation}
\langle i(\delta_1),i(\delta_2)\rangle_{\wti W'}=2\langle\delta_1,\delta_2\rangle^\el_W, \text{ for all }\delta_1,\delta_2\in R_\bC(W).
\end{equation}
\item Let $\delta\in\overline R_\bZ(W)$ be given.  Then 
\begin{equation}
i(\delta)=\wti\delta^+-\wti\delta^-,
\end{equation}
for unique $\wti W'$-representations $\wti \delta^+,\wti\delta^-$ such that 
\begin{equation}
\begin{aligned}
&\wti\delta^-=\Sg(\wti\delta^+),\ 
\langle\wti\delta^+,\wti\delta^-\rangle_{\wti W'}=0,\text{ and}
&\langle\wti\delta^+,\wti\delta^+\rangle_{\wti
  W'}=\langle\wti\delta^-,\wti\delta^-\rangle_{\wti
  W'}=\langle\delta,\delta\rangle^\el_W.
\end{aligned}
\end{equation}
 In particular, if $\delta$ is an element of norm one in $\overline R_\bZ(W)$, then $\wti\delta^+$, $\wti\delta^-$ are irreducible.

\end{enumerate}
\end{theorem}

\begin{proof}[Sketch of proof]
We sketch the proof for convenience. Claim (1) is immediate from (\ref{e:obs1}) and (\ref{e:obs2}).
For claim (2), one uses the involution $\Sg$ from (\ref{e:Sg}). Notice that
$$\Sg(i(\delta))=-i(\delta).$$
Since $i(\delta)$ is an integral virtual $\wti W'$-character, there exist $\wti W'$-representations $\wti\delta^+$ and $\wti\delta^-$ with $\langle\wti\delta^+,\wti\delta^-\rangle_{\wti W'}=0$, such that $i(\delta)=\wti\delta^+-\wti\delta^-.$ Applying $\Sg$, we conclude that $\wti\delta^-=\Sg(\wti\delta^+).$ The rest is a consequence of (1).
\end{proof}

Motivated by this result, we make the following definition.

\begin{definition}\label{d:irr0}
Let $\Irr_\gen^0\wti W'$ be the set of irreducible genuine $\wti W'$-representations $\wti\delta$ which occur as a component of $i(\delta)$ for some $\delta\in \overline R_\bZ(W)$, see Theorem \ref{t:CT}(2). From the construction, one sees that if $\wti\delta\in\Irr_\gen^0\wti W'$, then $\Sg(\wti\delta)\in\Irr_\gen^0\wti W'.$ 
\end{definition}

\begin{example} If $R$ is of type $A_{n-1}$, then $\Irr_\gen^0\wti S'_n=\{S^+,S^-\}$, where $S^+,S^-$ are the two associate spin $\wti S_n'$-modules. For comparison, in this case, $|\Irr_\gen\wti S_n'|=k_1+2k_2$, where $k_1$ is the number of partitions of $n$ of even length, and $k_2$ is the number of partitions of $n$  of odd length.

If $R$ is of type $B_n$,  then $\Irr_\gen^0\wti W'=\Irr_\gen\wti W'.$ This can be seen directly using the description of $\wti W(B_n)$-representations, see  \cite[Theorem 1.0.1 and section 3.3]{C}.

\end{example}

\subsection{The central character and Vogan's conjecture} Theorem \ref{t:vogan} imposes strict limitations on the central character of modules with nonzero Dirac cohomology. In light of Corollary \ref{c:imagezeta}, we may regard the homomorphism $\zeta$ as $$\zeta:Z(\bH)\to Z(\bC[\wti W']).$$

\begin{definition}\label{d:voganchar}
If $\wti\delta$ is an irreducible $\wti W'$-representation, define the homomorphism (central character) $\chi^{\wti\delta}: Z(\bH)\to \bC$ by $\chi^{\wti\delta}(z)=\wti\delta(\zeta(z))$ for every $z\in Z(\bH).$
\end{definition}

As in \cite{BCT}, an immediate corollary of Theorem \ref{t:vogan} is the following.

\begin{corollary}\label{c:voganchar}
Let $X$ be a $\bH$-module with central character $\chi_X$. Suppose there exists an irreducible $\wti W'$-representation $\wti\delta$ such that $\Hom_{\wti W'}[\wti\delta, H_D(X)]\neq 0.$ Then $\chi_X=\chi^{\wti\delta}.$
\end{corollary}

\subsection{The rank pairing} Let $K_0(\bH)$ be the Grothendieck group of finitely generated projective $\bH$-modules. Every such projective module $P$ is given by an idempotent $p\in M_N(\bH)$, for some natural number $N$, such that $P$ can be identified with the image of $p$ acting on $\bH^N.$ One defines the rank pairing to be the bilinear map
\begin{equation}\label{e:rkpair}
[~,~]: K_0(\bH)\times R_\bZ(\bH)\to \bZ,\quad \text{such that }[[P],[\pi]]:=\rank(\pi(p)),
\end{equation} 
for every finitely generated projective module $P$ and every irreducible $\bH$-module $(\pi,V_\pi);$ here $\pi(p)$ is regarded as an element of $M_N(\End_\bH(V_\pi)).$

The rank pairing extends to a Hermitian pairing $[~,~]:K_0(\bH)_\bC\times R_\bC(\bH)\to \bC.$ 

\begin{remark}\label{r:Wrkproj} 
If $\delta$ is a finite-dimensional $W$-representation, the induced module $\bH\otimes_{\bC[W]}\delta$, an $\bH$-module under left multiplication, is a finitely generated projective $\bH$-module. For every finite-dimensional $\bH$-module $X$, the rank pairing is
\begin{equation}\label{e:Wrkproj}
[\bH\otimes_{\bC[W]}\delta,X]=\dim \Hom_W(\delta, X).
\end{equation}
The same formula holds with $W'$ in place of $W.$
\end{remark}

We define a notion of support for elements of $K_0(\bH)_\bC$. If $\lambda\in V_\bC^\vee,$ let $\mathsf{Irr}_{W\lambda}(\bH)$ denote the set of isomorphism classes of irreducible $\bH$-modules with central character $W\lambda.$ 

\begin{definition}\label{d:support}
If $A\in K_0(\bH)_\bC$, define
\begin{equation}
\supp(A)=\{\lambda\in V_\bC^\vee: \text{there exists } X\in \mathsf{Irr}_{W\lambda}(\bH)\text{ such that } [A,X]\neq 0\}.
\end{equation}
Set $F_iK_0(\bH)_\bC=\{A\in K_0(\bH)_\bC: \dim\supp(A)\le i\}$ and  $F_{-1}K_0(\bH)_\bC=\cap_{\pi\in \mathsf{Irr}(\bH)}\ker[\cdot,\pi].$ Define
$$F_0H_0(\bH)=F_0K_0(\bH)_\bC/F_{-1}K_0(\bH)_\bC.$$
\end{definition}

By definition,
\begin{equation}\label{e:rknonleft}
[A,X]=0\text{ for all }X\in\mathsf{Irr}(\bH)\text{ implies that }A\in F_{-1}K_0(\bH)_\bC.
\end{equation}

\begin{lemma}\label{l:F0H0}
The rank pairing (\ref{e:rkpair}) gives rise to a canonical injective linear map
$\Phi: F_0H_0(\bH)\to \overline R_\bC(\bH)$ such that
\begin{equation}
[A,X]=\langle \Phi(A),X\rangle^\EP_\bH,\text{ for all } A\in F_0H_0(\bH),~ X\in \overline R_\bC(\bH).
\end{equation}
\end{lemma}

\begin{proof}
Firstly, the rank pairing descends to a pairing $[~,~]: F_0H_0(\bH)\times R_\bC(\bH)\to \bC,$ which is nondegenerate on the left by  (\ref{e:rknonleft}). Secondly, let $A\in F_0H_0(\bH)$ be given, and suppose that $X=\bH\otimes_{\bH_P}\sigma$ is a parabolically induced module, $\sigma\in\mathsf{Irr}(\bH_P)$ such that $[A,X]\neq 0.$ We can form the family $\{\sigma_\chi=\sigma\otimes\chi:\chi\text{ character of }Z(\bH_P)\}\subset\mathsf{Irr}(\bH_P)$. Then $[A,X_\chi]\neq 0$, for all $X_\chi= \bH\otimes_{\bH_P}\sigma_\chi$. But, by definition, this implies that the central characters of $X_\chi$ are all in $\supp(A)$, and therefore $\dim\supp(A)>0,$ a contradiction. Hence, $[A,X]=0$ for every proper parabolically induced  module $X$. 

This implies that the rank pairing descends to a pairing $F_0H_0(\bH)\otimes \overline R_\bC(\bH)\to \bC.$ Because this is nondegenerate on the left, it defines an injective linear map
\begin{equation}
\Phi': F_0H_0(\bH)\to \overline R_\bC(\bH)^*,\quad \Phi'(A)=[A,\cdot]: \overline R_\bC(\bH)\to \bC.
\end{equation}
Using the nondegenerate Hermitian pairing 
$\langle~,~\rangle^\EP_\bH: \overline R_\bC(\bH)\times \overline R_\bC(\bH)\to \bC,$
 we get an injection 
\begin{equation}
\iota: \overline R_\bC(\bH)\to \overline R_\bC(\bH)^*,\quad \iota(X)=\langle \cdot, X\rangle^\EP_\bH,\ X\in \overline R_\bC(\bH).
\end{equation}
Since $\overline R_\bC(\bH)$ is finite-dimensional, $\iota$ is a bijection. Define 
\begin{equation}
\Phi=\iota^{-1}\circ\Phi': F_0H_0(\bH)\to \overline R_\bC(\bH).
\end{equation}
It is clear from the definitions that $[A,X]=\langle \Phi(A),X\rangle^\EP_\bH,$ for all $A\in F_0H_0(\bH),$ $X\in \overline R_\bC(\bH)$.
\end{proof}

\subsection{The Dirac induction map $\Ind_D$}

\begin{definition}\label{d:indmap}
The Dirac induction map $\Ind_D:\overline R_\bC(W)\to \overline R_\bC(\bH)$ is the linear map defined on every $\delta\in\overline R_\bZ(W)$  by
\begin{equation}
\Ind_D(\delta)=\Phi([\bH\otimes_{W'}((\wti\delta^+)^*\otimes S^+)]-[\bH\otimes_{W'}((\wti\delta^+)^*\otimes S^-)]),
\end{equation}
where $\Phi$ is the map from Lemma \ref{l:F0H0} and $\delta^+$ is the $\wti W'$-representation from Theorem \ref{t:CT}(2). 
\end{definition}

\begin{theorem}\label{t:ind}
\begin{enumerate}
\item The map $\Ind_D$ from Definition \ref{d:indmap} is well-defined, i.e., $[\bH\otimes_{W'}((\wti\delta^+)^*\otimes S^+)]-[\bH\otimes_{W'}((\wti\delta^+)^*\otimes S^-)]\in F_0H_0(\bH).$
\item For every $\delta\in \overline R_\bZ(W),$ 
\begin{equation}
\langle\Ind_D(\delta),X\rangle^\EP_\bH=\langle\wti\delta^+,I(X)\rangle_{\wti W'},
\end{equation}
where $I(X)$ is the Dirac index of $X$.
\item For every $X\in \overline R_\bC(\bH)$ and $\delta\in \overline R_\bC(W),$ 
\begin{equation}
\langle\Ind_D(\delta),X\rangle^\EP_\bH=\langle\delta,\rr(X)\rangle^\el_W.
\end{equation}
\end{enumerate}
\end{theorem}

\begin{proof} Let $X$ be a finite-dimensional $\bH$-module. Let $\delta$ be an element of $\overline R_\bZ(\bH)$. Using (\ref{e:Wrkproj}), we have the rank pairings:
$$[\bH\otimes_{W'}((\wti\delta^+)^*\otimes S^\pm),X]=\dim\Hom_{W'}((\wti\delta^+)^*\otimes S^\pm,X)=\dim\Hom_{\wti W'}((\delta^+)^*,X\otimes (S^\pm)^*).$$
Therefore, using also Lemma \ref{l:CT}:
\begin{equation}\label{e:eq2}
\begin{aligned}
&[\bH\otimes_{W'}((\wti\delta^+)^*\otimes S^+),X]-[\bH\otimes_{W'}((\wti\delta^+)^*\otimes S^-),X]
&=\dim\Hom_{\wti W'}((\delta^+)^*,I(X)^*)\\
&=\langle\wti\delta^+,I(X)\rangle_{\wti W'}.\\
\end{aligned}
\end{equation}
Since $I(X)=0$ whenever $X$ is a proper parabolically induced module
(\ref{e:vanishinduced}), we see that the support of
$[\bH\otimes_{W'}((\wti\delta^+)^*\otimes
S^+),X]-[\bH\otimes_{W'}((\wti\delta^+)^*\otimes S^-),X]$ is a subset
of the set of central characters of elliptic tempered $\bH$-modules. Then (1) follows.

Claim (2) is immediate from (\ref{e:eq2}) and Lemma \ref{l:F0H0}.

For claim (3) let $\delta$ be an element of $\overline R_\bZ(W)$ and $X$ an element of $\overline R_\bZ(\bH)).$ Using Theorem \ref{t:CT}(1), 
\begin{equation}
\begin{aligned}
\langle\delta,X|_W\rangle^\el_W&=\frac 12\langle i(\delta),i(X|_W)\rangle_{\wti W'}=\frac 12\langle i(\delta),I(X)\rangle_{\wti W'}\\
&=\frac 12\langle \wti\delta^+-\wti\delta^-,I(X)\rangle_{\wti W'}
=\langle\wti\delta^+,I(X)\rangle_{\wti W'}.
\end{aligned}
\end{equation}
For the last equality, we used the involution $\Sg$ of (\ref{e:Sg}) to see that 
\begin{equation*}
\langle\wti\delta^+,I(X)\rangle_{\wti W'}=\langle\Sg(\wti\delta^+),\Sg(I(X))\rangle_{\wti W'}=\langle\wti\delta^-,-I(X)\rangle_{\wti W'}.
\end{equation*}
\end{proof}

\begin{definition}\label{d:pure} Let $(\C X,\langle~,~\rangle)$ be a $\bZ$-lattice in an Euclidean space $(E,\langle~,~\rangle)$. A vector $v\in \C X$ is called pure if $v$ cannot be written as a  sum $v=v_1+v_2$, where $v_1,v_2\in \C X\setminus \{0\}$ and $\langle v_1,v_2\rangle=0.$
\end{definition}

Denote
\begin{equation}\label{e:Y}
\C Y=\rr(\overline R_\bZ(\bH))\subset \overline R_\bZ(W),
\end{equation}
a $\bZ$-lattice with respect to the pairing $\langle~,~\rangle^\el_W$ in $\overline R_\bC(W).$ Notice that a priori, the lattice $\C Y$ depends on the parameter function $k$ of $\bH.$

\begin{corollary}\label{c:indmap}
\begin{enumerate}
\item $\Ind_D$ is the inverse map of $\rr$. 
\item For every $\delta\in \overline R_\bC(W)$, the Dirac
  index of $\Ind_D(\delta)$ is
\begin{equation}
I(\Ind_D(\delta))=i(\delta).
\end{equation}
In particular, $I(\Ind_D(\delta))$ is independent of the parameter function $k$ of $\bH$.
\item If $\delta$ is a pure element (in the sense of Definition \ref{d:pure})  of the lattice $\C Y$ defined in (\ref{e:Y}), then $\Ind_D(\delta)$ is supported by a single central character.
\item If a rational multiple of $\delta\in \overline R_\bZ(W)$ is pure in $\C Y$, then the central character of $\Ind_D(\delta)$ equals $\chi^{\wti\delta}$ (see Definition \ref{d:voganchar}) for any irreducible $\wti W'$-representation $\wti\delta$ occuring in $i(\delta)$.
\end{enumerate}
\end{corollary}

\begin{proof}
Claim (1) is immediate by Theorem \ref{t:ind}(3). Namely, given $X\in \overline R_\bC(\bH)$: 
\begin{equation}
\langle\Ind_D(\rr(X)),Y\rangle^\EP_\bH=\langle \rr(X),\rr(Y)\rangle^\el_W=\langle X,Y\rangle^\EP_\bH,
\end{equation}
for every $Y\in \overline R_\bC(\bH)$. Since $\langle~,~\rangle^\EP_\bH$ is nondegenerate on $\overline R_\bC(\bH)$, it follows that $\Ind_D(\rr(X))=X$ in $\overline R_\bC(\bH)$.

By Lemma \ref{l:CT}, $I(\Ind_D(\delta))=i(\rr(\Ind_D(\delta))).$ Claim
(2) now follows from (1).

For (3), the restriction $\Ind_D|_{\C Y}:\C Y\to \overline R_\bZ(\bH)$
is an isometric isomorphism and an inverse of $\rr: \overline
R_\bZ(\bH)\to \C Y.$ Given $\delta\in \C Y$, decompose
$\Ind_D(\delta)=X_1+\dots+X_\ell$ in $\overline R_\bZ(H)$, where:
\begin{enumerate}
\item[(a)] $X_i$ is in the class of an integral virtual $\bH$-module;
\item[(b)] $X_i$ has central character $\chi_i$;
\item[(c)] the central characters $\chi_i$ are mutually distinct.

\end{enumerate}
By (c), $\langle X_i,X_j\rangle^\EP_\bH=0$ for all $i\neq j.$ Then $\delta=\rr(\Ind_D(\delta))=\rr(X_1)+\dots+\rr(X_\ell)$ is an orthogonal decomposition of $\delta$ in $\C Y$. Since $\delta$ is pure, $\ell=1$ by definition.

To prove (4), notice that (3) implies that $\Ind_D(\delta)$ is supported by a single central character. On the other hand, if $\wti\delta$ occurs in $i(\delta)$, then $\langle \wti \delta,I(\Ind_D(\delta))\rangle_{\wti W'}\neq 0$. This means that $\wti\delta$ occurs in (one of) the Dirac cohomology groups $H_D(\Ind_D(\delta))$, so the claim follows from Corollary \ref{c:voganchar}.
\end{proof}

\begin{remark}
\begin{enumerate}
\item[(a)] Corollary \ref{c:indmap}(3) says in particular that if a rational multiple of $\delta\in \overline R_\bZ(W)$ is pure in $\C Y$, then for all $\wti W'$-irreducible constituents $\wti \delta$ of $i(\delta)$, the central characters $\chi^{\wti\delta}$ are the same. 
Define the central character of $\Ind_D$ to be
\begin{equation}
\cc(\Ind_D(\delta))=\chi^{\wti\delta},
\end{equation}
for every  $\delta\in\overline R_\bZ(W)$ which is a rational multiple of a pure element in $\C Y.$
\item[(b)] Every Euclidean lattice $L$ has a basis consisting of pure
  vectors. Recall that if $\C B$ is a basis for $L$, then the
  orthogonal defect of $\C B$ is defined as the ratio
  $\text{def}(\C B)=\prod_{x\in\C B}||x||/\text{vol}(L)$. 
%and that
%  $\text{def}(\C B)=1$ if and only if $\C B$ is orthogonal. 
Then every
  basis with minimal orthogonal defect consists of pure vectors.  

In our situation, this implies that $\C Y$ has a basis of pure
elements, and therefore that every element of $\overline R_\bZ(W)$ can
be written as a sum of elements of $\overline R_\bZ(W)$ which are
rational multiples of pure vectors in $\C Y$.

\item[(c)] We will improve Corollary \ref{c:indmap}(4) in Corollary \ref{c:puresupport}.
\end{enumerate}
\end{remark}

\subsection{The central character map} 
Let $\bH_A$ be the generic Hecke algebra over $A=\bC[\underline k]$ from Definition \ref{d:graded}.
\begin{definition}\label{d:residue}
Let $\Res^\lin(R)$ be the set of linear maps $\xi:\Spec(A)=\bC^{W\backslash
  R}\to V_\bC^\vee$ such that for almost all $k\in \bC^{W\backslash R}$, the point $\xi(k)$ satisfies the condition
\begin{equation}\label{e:bala}
\#\{\al\in R: \al(\xi(k))=k_\al\}=\#\{\al\in R: \al(\xi(k))=0\}+\dim V_\bC^\vee.
\end{equation}
\end{definition}

The set $\Res^\lin(R)$ is studied in \cite{HO,O,O2}, where it is shown that $\Res^\lin(R)$ is a nonempty, finite set, invariant under $W$, and its explicit description is given in all irreducible cases. Notice that when $k$ is a constant function, condition (\ref{e:bala}) is satisfied by the middle elements of distinguished Lie triples \cite{Ca}.

We will make repeated use of the following results. Assume the root
system $R$ does not have simply-laced factors and recall the notion
of generic parameter $k$ from \cite[section 4.3]{O} and
\cite[Definition 2.64]{OS1}.

\begin{theorem}[{\cite[Theorem 5.3, Definition 5.4, equation
    (79)]{OS1}}]\label{t:genericdiscrete} Suppose $R$ does not have simply laced factors, and
  let $\C Q$ be a generic region for the parameters $k$. 
\begin{enumerate}
\item If $\pi$ is
  an irreducible discrete series $\bH_k$-module, there exists
  $\lambda\in \Res^\lin(R)$ such that the central character of $\pi$
  is $\lambda(k).$

\item Conversely, if $\lambda\in\Res^\lin(R)$, for every $k\in\C Q$ there
exists a unique irreducible discrete series $\bH_k$-module $\pi_k$
with central character $\lambda(k).$

\end{enumerate}
\end{theorem}

Combining Theorem \ref{t:genericdiscrete} with the results of
\cite{OS} on the Euler-Poincar\'e pairing as recalled in sections
\ref{s:EP} and \ref{s:EPH}, one has the following result.

\begin{theorem}[\cite{OS,OS1}]\label{t:EPgeneric}
Suppose $R$ does not have simply laced factors, and $k$ is a generic
parameter for $\bH$. Then the set of irreducible discrete series
modules $\DS(\bH)$ is an orthonormal basis for $\overline R_\bZ(\bH).$
\end{theorem}

Motivated by the results of the previous section, let $\C S(\overline
R_\bZ(W),\C Y)$ be the set of elements $\delta\in \overline R_\bZ(W)$ which are rational multiples of pure elements of $\C Y$. Define
\begin{equation}\label{e:ccmap1}
\Lambda:\C S(\overline R_\bZ(W),\C Y)\to \Spec(Z(\bH_A)),\ 
\Lambda(\delta)(k)=\cc(\Ind_{D,k}(\delta))\in
\Spec(Z(\bH_k))=W\backslash V_\bC^\vee.
\end{equation}
Using Definition \ref{d:irr0} and Corollary \ref{c:indmap}(4), we see that
\begin{equation}
\im\Lambda=\{\chi^{\wti\delta}: \wti\delta\in\Irr^0_\gen(\wti W')\}.
\end{equation}

\begin{theorem}\label{t:ccmap}
\begin{enumerate}
\item
Let $\wti\delta\in \Irr_\gen^0(\wti W')$ be given. Then $\chi^{\wti\delta}$ is linear in $\underline k$, i.e., there exists a linear function $\xi(\wti\delta):\Spec(A)\to V_\bC^\vee$ such that $\chi^{\wti\delta}=W\cdot\xi(\wti\delta).$
\item $\Res^\lin(R)$ is contained in $\im\Lambda$,  with equality if $R$ has no simply-laced factors.
\end{enumerate}
\end{theorem}

\begin{proof}
(1) Let $\delta\in\overline R_\bZ(W)$ be such that $\wti\delta$ is a
component of $i(\delta).$ Since $\C Y$ admits a basis of pure vectors,
we may assume that $\delta$ is pure. The equality $\Lambda(\delta)=\chi^{\wti\delta}$ was verified in Corollary \ref{c:indmap}(4). It is sufficient to prove the linearity of $\Lambda(\delta)$ in the case when $R$ is irreducible. 

Assume that $R$ is simply laced. The generic algebra $\bH_A$ admits scaling isomorphisms $s_c:\bH_k\to \bH_{ck}$, for a scalar $c$, given by $s_c(w)=w$, $w\in W$ and $s_c(v)=cv,$ for $v\in V$. Notice that the Dirac element $\C D$ is unchanged under $s_c.$ For a central element $z\in S(V_\bC)^W$ homogeneous of degree $N$, equation (\ref{e:defzeta}) shows that $\zeta_{ck}(z)$ has degree $N$ in $c$. Therefore, 
$$\zeta_{ck}(z)=c^N\zeta_k(z),\text{ and in particular }\zeta_k(z)=k^N\zeta_1(z).$$
Hence $\chi^{\wti\delta}$ is linear in $k$ for all irreducible $\wti W'$-representations $\wti\delta.$

Now, let $R$ be a simple root system which is not simply laced.
%By Theorem \ref{t:vogan}, the map $\Lambda(\delta)$ can be
%regarded as an algebra homomorphism $$\Lambda(\delta): Z(\bH_A)\to
%A,\quad z\mapsto \Lambda(\delta)(z).$$
% In other words, $\Lambda(\delta)$ is a $\bC[k]$-valued point in $W\backslash V_\bC^\vee.$
Fix a generic region $\C Q$ of the parameter function $k$, and
$k_0\in\C Q.$
Corollary \ref{c:indmap}(4) and Theorem \ref{t:genericdiscrete} imply that there exists $\lambda\in\Res^\lin(R)$ such that 
$\Lambda(\delta)(k)=\lambda(k)$ for every $k\in\C Q.$ Consider the formal completion $\widehat{\bC[k]}$ of $\bC[k]$ at $k_0,$ and $i_{k_0}:\bC[k]\to\widehat{\bC[k]}$ the canonical map. Since $$\Lambda(\delta)\circ  i_{k_0}=\lambda\circ i_{k_0},$$
as homomorphisms $S(V_\bC)^W\to \widehat{\bC[k]}$, it follows from the injectivity of $i_{k_0}$ that $\Lambda(\delta)=\lambda$.

\

(2) When $k$ is constant (more generally when $k$ is of geometric origin in the sense of \cite{L2}), this is (part of) \cite[Theorem 5.8]{BCT}, which is a  corollary of Vogan's conjecture \ref{t:vogan} together with \cite[Theorems 1.0.1 and 3.10.3]{C}. In particular, this is the case when $R$ is simply-laced. The explicit map $\Lambda$ for these cases can be found in the tables of \cite{C}.

Assume now that $R$ is a simple root system which is not simply laced. By (1), $\Lambda(\delta)$ gives a $W$-orbit of linear maps $\bC^{(W\backslash R)}\to V_\bC^\vee$. So it remains to check condition (\ref{e:bala}) in one generic region of $k$.
To simplify notation, let $k$ be the parameter on the long roots of $R$ and $k'$ the parameter on the short roots.

When $R$ is of type $B_n$, we choose the generic region $k'/k>n-1.$ To every partition $\sigma$ of $n$, \cite{O} attaches an element $\cc_\sigma$ of $\Res^\lin(R)$, and every $W$-conjugacy class in $\Res^\lin(R)$ contains one and only one such $\cc_\sigma.$ Let $\sigma\times\emptyset$ denote the irreducible $W(B_n)$-representation obtained by inflating to $W(B_n)$ the irreducible $S_n$-representation given in Young's parametrization by $\sigma.$ An easy algebraic argument (see \cite[section 4.7]{CK}) shows that there is a unique discrete series module $\pi_\sigma$ with central character $\cc_\sigma$ and $\pi_\sigma|_W\cong (\sigma\times\emptyset)\otimes\sgn.$ It is well-known (see \cite{Rea}) that $(\sigma\times\emptyset)\otimes S$ is an irreducible $\wti W$-representation, for each spin $\wti W$-module $S$, and in fact, every genuine irreducible $\wti W$-representation is obtained in this way. Fix a spin module $S$ and denote $\wti\sigma=(\sigma\times\emptyset)\otimes S\otimes\sgn$. Then
\begin{equation}
I(\pi_\sigma)=\pi_\sigma|_W\otimes (S^+-S^-)=(\sigma\times\emptyset)\otimes (S^+-S^-)\otimes\sgn\neq 0.
\end{equation}
Therefore, $\Hom_{\wti W}[\wti\sigma,H_D(\pi_\sigma)]\neq 0,$ and by Corollary \ref{c:voganchar}, $\chi^{\wti\sigma}=\cc_\sigma.$ This completes the proof in type $B_n.$
 
For types $G_2$ and $F_4$, the notation for genuine $\wti W$-types is as in the character tables of \cite{M}. By (\ref{e:zetaomega}), if $\Lambda(\delta)=\xi\in \Res^\lin(R)$, then $$\langle\xi(k,k'),\xi(k,k')\rangle=\wti\delta(\Omega_{\wti W,k,k'}),$$
as functions in $k,k'.$ 
It turns out that for $F_4$ and $G_2,$ if $\xi\neq\xi'\in W\backslash \Res^\lin(R)$, then $\langle\xi(k,k'),\xi(k,k')\rangle\neq \langle\xi'(k,k'),\xi'(k,k')\rangle$ as functions in $k,k'$. (This is not the case for type $B_n$ discussed above, when $n$ is large.) Therefore, to identify the images of the map $\Lambda$, it is sufficient in this case to compute the scalar functions $\wti\delta(\Omega_{\wti W,k,k'})$ and compare. The results are tabulated in Tables \ref{ta:F4} and \ref{ta:G2}.

\begin{table}[h]
\caption{$\Res^\lin(F_4)$\label{ta:F4}}
\begin{tabular}{|c|c|}
\hline
central character &$\wti\delta\in \Irr_\gen^0(\wti W)/_\sim$\\

\hline
$k\om_1+k\om_2+k'\om_3+k'\om_4$ &$4_s$\\
\hline
$k\om_1+k\om_2+(-k+k')\om_3+k'\om_4$ &$8_{sss}$\\
\hline
$k\om_1+k\om_2+(-k+k')\om_3+k\om_4$ &$12_s$\\
\hline
$k\om_1+k\om_2+(-2k+k')\om_3+k'\om_4$ &$4_{ss}$\\
\hline
$k\om_1+k\om_2+(-2k+k')\om_3+2k\om_4$ &$24_s$\\
\hline
$k\om_1+k\om_2+(-2k+k')\om_3+k\om_4$ &$12_{ss}$\\
\hline
$k\om_1+k\om_2+(-2k+k')\om_3+(3k-k')\om_4$ &$8_{ssss}$\\
\hline
$k\om_2+(-k+k')\om_4$ &$8_{ss}$, $8_s$\\
\hline
\end{tabular}
\end{table}

\begin{table}[h]
\caption{$\Res^\lin(G_2)$\label{ta:G2}}
\begin{tabular}{|c|c|}
\hline
central character &$\wti\delta\in \Irr_\gen^0(\wti W)/_\sim$\\

\hline
$k\om_1+k'\om_2$ &$2_s$\\
\hline
$k\om_1+(-k+k')\om_2$ &$2_{ss}$\\
\hline
$k\om_1+\frac 12 (-k+k')$ &$2_{sss}$\\
\hline
\end{tabular}
\end{table}

\end{proof}

\section{Orthogonal bases for spaces of virtual elliptic
  characters}\label{s:ortho} We would like to describe the map
$\Ind_D$ from Definition \ref{d:indmap} explicitly and study its
integrality properties. For this, we will show that the lattice
$\overline R_\bZ(W)$ admits an orthogonal basis for all irreducible
root systems $R$. In particular, this basis consists of pure vectors,
in the sense of Definition \ref{d:pure}.

In Theorem \ref{t:Honbasis}, we determine orthogonal bases for the
spaces of virtual elliptic characters $\overline R_\bZ(\bH)$, for
every irreducible root system $R$ and every parameter function $k$,
except when $R=F_4$ and $k$ is special nonconstant.  By Corollary
\ref{c:resmap}, we see that every orthogonal basis of $\overline
R_\bZ(\bH)$ with respect to $\langle~,~\rangle^\EP_\bH$ gives, by
restriction to $W$, an orthogonal basis of $\overline R_\bZ(W)$ with respect to $\langle~,~\rangle^\el_W.$ 

When $R$ is not simply laced and the parameter function is generic, Theorem \ref{t:EPgeneric} gives orthonormal bases of $\overline R_\bZ(\bH)$. For special
parameters (i.e., non-generic) we use a limiting argument. This
argument is available when $R$ is of type $B_n$ via the exotic
geometric models for $\bH$-modules of \cite{K}, and it was already used
in \cite{CKK}. This approach can also be applied when $R$ is of type
$D_n$, since the graded Hecke algebra of type $B_n$ with parameter $0$
on the short roots is a $\bZ/2\bZ$ extension of the graded Hecke
algebra of type $D_n.$ However, in the $D_n$ case, one needs to
analyze carefully the changes in the R-groups under the extension by
$\bZ/2\bZ$, using the explicit description of type $B_n$ R-groups in
\cite{Sl}. The same limit argument can be used for $R=G_2$, where it
easy to construct explicit models for the families of discrete series modules
in the generic regions. Thus the cases that we cannot treat here are
certain special values of the parameters when $R=F_4.$

When $R$ is simply laced, especially if $R$ is of type $E$, we need to use the geometric classification of \cite{KL} and the results in \cite{R} relating the elliptic theories with the geometry of Kazhdan and Lusztig. Let $\fk g$ be the complex simple Lie algebra with Cartan subalgebra $V^\vee_\bC$ and root system $R$ and $G$ be the complex connected adjoint Lie group with algebra $\fk g.$ By \cite{KL,L,L2}, when the parameter function is constant $k=1,$ the irreducible $\bH$-modules are parametrized by $G$-conjugacy classes of triples $(s,e,\psi)$, where $s\in \fg$ is semisimple, $e\in \fg$ is such that $[s,e]=e$, in particular, $e$ is nilpotent, and $\psi$ is an irreducible representation of Springer type of the group of components $A(s,e)$ of the centralizer in $G$ of $s$ and $e$. Write $\pi_{(s,e,\psi)}$ for the module parametrized by the class $[(s,e,\psi)].$ In this correspondence, tempered modules with real central character are attached to triples $(\frac 12 h,e,\psi)$,  for a Lie triple $(e,h,f)$, while discrete series modules are attached to $(\frac 12 h,e,\psi)$ where $e$ is distinguished in the sense of Bala-Carter \cite{Ca}. Thus tempered modules with real central character are uniquely determined by $e$ and $\psi\in\widehat A(e)_0$ (characters of the component group of Springer type), so we write $\pi_{e,\psi}$ in place of $\pi_{(\frac 12 h,e,\psi)}.$ 

Reeder's results \cite{R} imply in this case that an irreducible tempered module $\pi_{(e,\psi)}\equiv 0$ in $\overline R_\bZ(\bH)$, unless $e$ is a quasidistinguished nilpotent element (see \cite[(3.2.2)]{R} for the definition, recalled below). If the parameter function $k$ is constant $k=1$, define the following set, consisting of elliptic tempered modules,
\begin{equation}\label{e:ETbasis}
\C B(\overline R_\bZ(\bH))=\{\pi_{e,\psi}: e\text{ dist.},~\psi\in\widehat A(e)_0\}\cup\{\pi_{e,\triv}: e\text{ quasidist., not dist.}\}.
\end{equation}
The notation for nilpotent orbits below is as in \cite{Ca}.

\begin{theorem}\label{t:Honbasis}
\begin{enumerate}
\item[(1)] Assume the root system $R$ is not of type $D_{2n}$ or
  $E_7$, and if $R=F_4$ then the parameter function $k$ is assumed
  either constant or generic. The space $\overline R_\bZ(\bH)$ has an orthonormal basis with respect to $\langle~,~\rangle^\EP_\bH$ consisting of elliptic tempered $\bH$-representations. In particular, when the parameter function is constant $k=1$, the set $\C B(\overline R_\bZ(\bH))$ from (\ref{e:ETbasis}) is such an orthonormal basis.
\item[(2)] Suppose $R=D_{2n}.$ The set  $\C B(\overline R_\bZ(\bH))$ is an orthogonal basis of $\overline R_\bZ(\bH)$ such that every element is a unit element, except the irreducible tempered modules $\pi_{e,\triv}$  when $e$ are representatives of quasidistinguished nilpotent orbits labeled by partitions $(a_1,a_1,a_2,a_2,\dots,a_{2l},a_{2l})$ of $2n$, $0<a_1<a_2<\dots<a_{2l}$,  which have elliptic norm $\sqrt 2$. 
\item[(3)] Suppose $R=E_7.$ The set 
$\C B(\overline R_\bZ(\bH))$
is an orthogonal basis for $\overline R_\bZ(\bH)$ such that every element is a unit element, except the irreducible tempered module $\pi_{A_4+A_1,\triv}$ which  has elliptic norm $\sqrt 2$.
\end{enumerate}

\end{theorem}

The proof of Theorem \ref{t:Honbasis} is case by case and it is presented in subsections \ref{s:nonsimplylaced}-\ref{s:typeE}. The case $R=A_{n-1}$ is well-known: the space $\overline R_\bZ(\bH)$ is one-dimensional spanned by the Steinberg module.

\subsection{$R$ not simply laced}\label{s:nonsimplylaced} When the parameter function $k$ is
generic, Theorem \ref{t:EPgeneric} says that $\DS(\bH)$ is an
orthonormal basis of $\overline R_\bZ(\bH)$, and therefore this proves
Theorem \ref{t:Honbasis}(1) in this case.

When $k$ is non-generic, one proves the result by a  known limiting
argument, e.g., \cite[section 2.4]{CK}. Assume $R$ is of type $B_n$ or $G_2.$ Let $k_0$ be a special parameter function. Without loss of generality, we may assume that $k_0(\al)=1$, when $\al$ is a long root, and $k_0(\beta)=m_0$, when $\beta$ is a short root. Then there exists $\ep>0$ such that the parameter function $k_t$, where $k_t(\al)=1$ and $k_t(\beta)=m_0+t$, is generic for all $t\in (0,\ep).$  Let $$\C F=\{\pi^{\C F}_t: \pi^{\C F}_t\in \DS(\bH_{k_t})\}_{0<t<\ep}$$ be a continuous family of discrete series modules as in \cite[Definition 3.5]{OS1}. Then $\pi^{\C F}_t|_W\cong \pi^{\C F}_{t'}|_W$ for all $t,t'\in (0,\ep).$ One can consider the limit module $\pi^{\C F}_0=\lim_{t\to 0^+}\pi^{\C F}_t$. This is a tempered $\bH_{k_0}$-module with the same $W$-structure as $\pi^{\C F}_t.$ In particular, the set $\{\pi^{\C F}_0|_W: \C F\}$ is orthonormal in $ \overline R_\bZ(W)$, and again by the isometry $r$ from (\ref{e:rmap}), the set $\C B=\{\pi^{\C F}_0: \C F\}$ is orthonormal in $\overline R_\bZ(\bH)$. By \cite[Theorem A]{CKK}, when $R$ is of type $B_n$ this set consists of irreducible $\bH_{k_0}$-modules. Corollary \ref{c:resmap} implies in particular that $\dim \overline R_\bZ(\bH_{k_0})=\dim \overline R_\bZ(\bH_{k_t})$, and thus $\C B$ is a basis of $\dim \overline R_\bZ(\bH_{k_0}).$

\subsection{$R=D_n$} We begin by recalling the definition of quasidistinguished nilpotent elements and presenting their classification when $\fg$ is classical. Let $e$ be a nilpotent element in $\fg$, and assume that $e$ is contained in a Levi subalgebra $\fk m$ of $\fg$ with Levi subgroup $M$ of $G$. The natural map of component groups $A_M(e)\to A_G(e)$ is in fact an injection. The element $e$ is called quasidistinguished if 
\begin{equation}
A_G(e)\neq \bigcup_{e\in \fk m} A_M(e),
\end{equation}
where the union in the right hand side is over all proper Levi subalgebras containing $e$. Every distinguished nilpotent element is automatically quasidistinguished. The classification of nilpotent adjoint orbits when $\fg$ is of classical type is based on the Jordan canonical form. The nilpotent orbits are parametrized by:
\begin{enumerate}
\item[(i)] partitions of $n$, when $\fg=sl(n)$;
\item[(ii)] partitions of $2n$ where every odd part occurs with even multiplicity, when $\fg=sp(2n)$;
\item[(iii)] partitions of $m$ where every even part occurs with even multiplicity, when $\fg=so(m)$, except there are two distinct nilpotent orbits in $so(2n)$ for every partition where all parts are even. 
\end{enumerate}
The distinguished orbits are parametrized by:
\begin{enumerate}
\item[(i)] the partition $(n)$ (principal nilpotent orbit), when $\fg=sl(n)$;
\item[(ii)] partitions of $2n$ of the form $(a_1,a_2,\dots,a_l)$, where $0<a_1<a_2<\dots<a_l$ are all even, when $\fg=sp(2n)$;
\item[(iii)] partitions of $m$ of the form $(a_1,a_2,\dots,a_l)$, where $0<a_1<a_2<\dots<a_l$ are all odd, when $\fg=so(m)$.
\end{enumerate}
The centralizers of nilpotent elements $e$ and the component groups $A(e)$ are known explicitly for classical group, they were computed by Springer and Steinberg, see \cite[pp.398-399]{Ca}. A case-by-case analysis leads to the following classification of quasidistinguished orbits.

\begin{lemma}
\begin{enumerate}
\item[(i)] If $\fg=sl(n)$, the only quasidistinguished nilpotent orbit is the principal one.
\item[(ii)] If $\fg=sp(2n)$, the quasidistinguished nilpotent orbits are labeled by partitions of $2n$ with even parts such that the multiplicity of every part is at most two.
\item[(iii)] If $\fg=so(2n+1)$,   the quasidistinguished nilpotent orbits are labeled by partitions of $2n+1$ with odd parts such that the multiplicity of every part is at most two.
\item[(iv)] If $\fg=so(2n)$,   the quasidistinguished nilpotent orbits are labeled by partitions of $2n$ with odd parts, such that the multiplicity of every part is at most two, and if there are no parts with multiplicity one, then the number of distinct odd parts is even.
\end{enumerate}
\end{lemma}

In light of Corollary \ref{c:EPds}, we only need to compute the elliptic norms of tempered modules $\pi_{e,\psi}$, when $e$ is quasidistinguished, but not distinguished. By Corollary \ref{c:resmap} and \cite[Proposition 3.4.3]{R} 
\begin{equation}\label{e:Reeder}
\langle\pi_{e\psi},\pi_{e,\psi'}\rangle^\EP_\bH=\langle\pi_{e,\psi}|_W,\pi_{e,\psi'}|_W\rangle^\el_W=\langle\psi,\psi'\rangle^\el_{A(e)}
\end{equation}
where the elliptic pairing in the right hand side of (\ref{e:Reeder}) is with respect to the action of $A(e)$ on the toral Lie algebra $\fk s_0$ of the reductive centralizer of $e$  in $G$, \cf \cite[section 3.2]{R}. Denote this latter elliptic space by $\overline R_\bZ(A(e)).$

\begin{lemma}Let $R$ be of type $D_n$ and $e$ be a quasidistinguished, not distinguished nilpotent element in $\fg=so(2n)$, parameterized by a partition $\tau$ of $2n$. 
\begin{enumerate}
\item[(a)] If $\tau=(a_1,a_1,a_2,a_2,\dots,a_l,a_l,b_1,b_2,\dots,b_{2k}),$ where $k\ge 1,$ $0<a_1<a_2<\dots<a_l,$  $0<b_1<b_2<\dots<b_{2k}$, $a_i\neq b_j,$ for all $i,j$, and all $a_i$ and $b_j$ are odd, then $\dim \overline R_\bZ(A(e))=1$, and $\langle\triv,\triv\rangle^\el_{A(e)}=1.$
\item[(b)] If $\tau=(a_1,a_1,a_2,a_2,\dots,a_{2l},a_{2l}),$ where $0<a_1<a_2<\dots<a_{2l},$ and all $a_i$ are odd, then $\dim \overline R_\bZ(A(e))=1$, and $\langle\triv,\triv\rangle^\el_{A(e)}=2.$
\end{enumerate}

\end{lemma}

\begin{proof} The claims follow immediately once we describe explicitly the action of $A(e)$ on $\fk s_0.$

(a) In this case, the component group is $A(e)=(\bZ/2\bZ)^l\times (\bZ/2\bZ)^{2k-2}$ acting on the space $\fk s_0$ of dimension $l$. The action is as follows: $(\bZ/2\bZ)^{2k-2}$ acts trivially, while $(\bZ/2\bZ)^l$ acts via $$\bigoplus_{i=1}^l \triv\boxtimes\dots\boxtimes\triv\boxtimes\sgn_i\boxtimes\triv\boxtimes\dots\boxtimes\triv,$$ where $\sgn_i$ is the $\sgn$ representation on the $i$-th position. 

(b) In this case, the component group is $A(e)=(\bZ/2\bZ)^{2l-1}$ acting on the space $\fk s_0$ of dimension $2l$. The action is  via $$2\bigoplus_{i=1}^{2l-1} \triv\boxtimes\dots\boxtimes\triv\boxtimes\sgn_i\boxtimes\triv\boxtimes\dots\boxtimes\triv.$$
\end{proof}

\subsection{$R$ of type $E$}\label{s:typeE}
It remains to discuss the cases when $R$ is of type $E_6,E_7,E_8$.

In type $E_8$, there are four quasidistinguished, non-distinguished nilpotent orbits, and in all cases $A(e)\cong \bZ/2\bZ$ acts by the $\sgn$ representation on $\fk s_0$. This means that $\dim \overline R_\bZ(A(e))=1$ and $\langle\triv,\triv\rangle^\el_{A(e)}=1$ in all cases. The explicit relations between tempered modules are:
\begin{enumerate}
\item[(a)] $D_5+A_2$: $\pi_{D_5+A_2,\triv}\oplus\pi_{D_5+A_2,\sgn}=\ind_{\bH(D_5+A_2)}^{\bH(E_8)}(\pi_{(91),\triv}\boxtimes\pi_{(3),\triv})$;
\item[(b)] $D_7(a_1)$: $\pi_{D_7(a_1),\triv}\oplus\pi_{D_7(a_1),\sgn}=\ind_{\bH(D_7)}^{\bH(E_8)}(\pi_{(11,3),\triv})$;
\item[(c)] $D_7(a_2)$: $\pi_{D_7(a_2),\triv}\oplus\pi_{D_7(a_2),\sgn}=\ind_{\bH(D_7)}^{\bH(E_8)}(\pi_{(9,5),\triv})$;
\item[(d)] $E_6(a_1)+A_1$: $\pi_{E_6(a_1)A_1,\triv}\oplus\pi_{E_6(a_1)A_1,\sgn}=\ind_{\bH(E_6+A_1)}^{\bH(E_8)}(\pi_{E_6(a_1),\triv}\boxtimes \pi_{(2),\triv})$.
\end{enumerate}

In type $E_7$, there are two quasidistinguished, non-distinguished nilpotent orbit and in both cases $A(e)\cong \bZ/2\bZ$:
\begin{enumerate}
\item $E_6(a_1)$: $\pi_{E_6(a_1),\triv}\oplus\pi_{E_6(a_1),\sgn}=\ind_{\bH(E_6)}^{\bH(E_7)}(\pi_{E_6(a_1),\triv})$; the group $A(e)$ acts by $\sgn$ on $\fk s_0$;
\item $A_4+A_1$: $\pi_{A_4+A_1,\triv}\oplus \pi_{A_4+A_1,\sgn}=\ind_{A_4+A_1}^{E_7}(\pi_{(5)}\boxtimes\pi_{(2)})$. Here $\fk s_0$ is two-dimensional, and the group $A(e)=\bZ/2\bZ$ acts by $2\sgn$ on $\fk s_0$. This has the effect that $\wedge^\pm \fk s_0=2(\triv-\sgn)$, and therefore $\dim \overline R_\bZ(A(e))=1$, but $\langle\triv,\triv\rangle^\el_{A(e)}=2.$
\end{enumerate}

In type $E_6$, there is one quasidistinguished, non-distinguished nilpotent orbit denoted $D_4(a_1)$, whose component group is $S_3.$ There are three tempered modules $\pi_{D_4(a_1),\triv}$, $\pi_{D_4(a_1),(21)}$, and $\pi_{D_4(a_1),\sgn}$. First we have $\ind_{\bH(D_4)}^{\bH(E_6)}(\pi_{(53),\triv})=\pi_{D_4(a_1),\triv}\oplus 2\pi_{D_4(a_1),(21)}\oplus\pi_{D_4(a_1),\sgn}.$ Next, in $D_5,$ the induced module $\ind_{\bH(D_4)}^{\bH(D_5)}(\pi_{(53),\triv})$ splits into a sum of two tempered modules $\pi_{(5311),\triv}\oplus\pi_{(5311),\sgn}$, and we have
\begin{equation}
\begin{aligned}
\ind_{\bH(D_5)}^{\bH(E_6)}(\pi_{(5311),\triv})&=\pi_{D_4(a_1),\triv}\oplus \pi_{D_4(a_1),(21)},\\
 \ind_{\bH(D_5)}^{\bH(E_6)}(\pi_{(5311),\sgn})&=\pi_{D_4(a_1),(21)}\oplus\pi_{D_4(a_1),\sgn}.
\end{aligned}
\end{equation}
The group $A(e)=S_3$ acts on the two-dimensional $\fk s_0$ via the reflection representation, therefore $\dim \overline R_\bZ(A(e))=1$ and $\langle\triv,\triv\rangle^\el_{A(e)}=1.$

\bigskip

This concludes the proof of Theorem \ref{t:Honbasis}.

\subsection{Dirac indices}We end the section with the calculation of Dirac indices in the cases when the basis elements in $\C B(\overline R_\bZ(\bH))$ have elliptic norm $\sqrt 2.$

\begin{example}\label{ex:E7.1}
Assume $R=E_7$ and $\delta=\pi_{A_4+A_1,\triv}|_W\in \overline R_\bZ(W).$ Then there exist genuine $\wti W$-representations $\wti\delta^+,\wti\delta^-$, $\wti\delta^-=\wti\delta^+\otimes\sgn$ such that $i(\delta)=\wti\delta^+-\wti\delta^-.$ In this case, $\wti\delta^+$ is the sum of two irreducible $64$-dimensional $\wti W$-representations. 
\end{example}

\begin{proof}
The first claim follows from Theorem \ref{t:CT}(2) since $\langle\delta,\delta\rangle^\el_W=2$. The two irreducible $\wti W$-representations that enter are explicitly known by \cite[Table 4]{C}, where they are denoted by $64_s$ and $64_{ss}.$

The index $i(\delta)=I(\pi_{A_4+A_1,\triv})$ can also be calculated more directly as follows. Firstly, by Corollary \ref{c:voganchar} and (\ref{e:zetaomega}), the only irreducible $\wti W$-representations $\wti\sigma$ that can contribute to $I(\pi_{A_4+A_1,\triv})$ have the property that $\wti\sigma(\Omega_{\wti W})=\langle h/2,h/2\rangle,$ where $h$ is a middle element of a Lie triple for the nilpotent orbit $A_4+A_1.$ In \cite[Table 4]{C}, it is calculated that there are only four $\wti W$-representations with this property: $64_s,$ $64_{ss}$, $64_s\otimes\sgn$ and $64_{ss}\otimes \sgn.$ Next, one needs to see which of them occur in $\pi_{A_4+A_1,\triv}\otimes S^\pm$. As it is well-known, $\pi_{A_4+A_1,\triv}|_W\cong H^\bullet(\mathbf B_e)^\triv\otimes \sgn,$ where $e$ is a nilpotent element of type $A_4+A_1$, and $H^\bullet(\mathbf B_e)$ is the total cohomology of the Springer fiber $\mathbf B_e.$ The structure of $ H^\bullet(\mathbf B_e)^\triv$ is known explicitly by \cite{BS}, but we will not need this. The top degree component is the $W(E_7)$-representation denoted by $\phi_{512,11}$ in \cite{Ca}. Thus, Springer theory tells us that
\begin{equation}
\dim\Hom_{W(E_7)}(\phi_{512,12},\pi_{A_4+A_1,\triv})=1.
\end{equation}
(Again, in the notation of \cite{Ca}, $\phi_{512,12}=\phi_{512,11}\otimes\sgn$.) A direct calculation using the software package \texttt{chevie} in the computer algebra system GAP3.4 and the character table for $\wti W(E_7)$ from \cite{M}, reveal that
\begin{equation}\label{e:tensE7}
64_s\otimes S^+=64_{ss}\otimes S^+=\phi_{512,12}.
\end{equation}
(Notice that $\dim S^+=2^3$ indeed.) But this means that 
\begin{equation}
\dim\Hom_{\wti W(E_7)}(64_s,\pi_{A_4+A_1,\triv}\otimes S^+)=\dim\Hom_{\wti W(E_7)}(64_{ss},\pi_{A_4+A_1,\triv}\otimes S^+)=1.
\end{equation}
Since $\pi_{A_4+A_1,\triv}$ is a unitary $\bH$-module, \cite[Proposition 4.9]{BCT} implies that in fact
\begin{equation}
\ker D^+=64_s+64_{ss}; 
\end{equation}
similarly, one deduces that $\ker D^-=64_s\otimes\sgn+64_{ss}\otimes \sgn$, and the claim follows.
\end{proof}

\begin{example}\label{ex:D2n.1}
Assume $R=D_{2n}$,  and $\tau=(a_1,a_1,a_2,a_2,\dots,a_{2l},a_{2l})$ is a partition of $2n$, where $0<a_1<a_2<\dots<a_{2l},$ and all $a_i$ are odd. Then there exist genuine $\wti W'$-representations $\wti\delta^+,\wti\delta^-$, $\wti\delta^-=\Sg(\wti\delta^+)$ such that $i(\delta)=\wti\delta^+-\wti\delta^-.$ In this case, $\wti\delta^+$ is the sum of two irreducible $\wti W'$-representations. 
\end{example}

\begin{proof}
 To determine $\wti\delta^+_1$ and $\wti\delta^+_2$ explicitly, we rely on the classification of genuine $\wti W(D_m)$-modules \cite{Rea}, and the Dirac cohomology calculations in \cite{C}. Recall that a complete set of inequivalent genuine irreducible $\wti W(B_m)$-modules is given by $\{(\sigma\times\emptyset)\otimes \C S: \sigma\text{ partition of }m\}$, where $\C S\in\{S^+,S^-\}$, when $m$ is odd, and $\C S=S$, when $m$ is even. The restriction to $\wti W(D_m)$ yields the following complete sets of inequivalent irreducible modules:
\begin{enumerate}
\item[(i)] $\{(\sigma\times\emptyset)\otimes S^+: \sigma\text{ partition of }m\}$, when $m$ is odd. In this case, $(\sigma\times\emptyset)\otimes S^-=(\sigma^t\times\emptyset)\otimes S^+$ as $\wti W(D_m)$-representations.
\item[(ii)] $\{(\sigma\times\emptyset)\otimes S: \sigma\text{ partition of }m,\ \sigma^t\neq \sigma\}\cup \{\wti\sigma_1,\wti\sigma_2: \sigma\text{ partition of }m,\ \sigma^t=\sigma\}$, when $m$ is even, where $(\sigma\times\emptyset)\otimes S|_{W(D_m)}=\wti\sigma_1\oplus\wti\sigma_2,$ when $\sigma^t=\sigma.$ In this case, every irreducible $\wti W(D_m)$-representation is self dual under tensoring with $\sgn.$
\end{enumerate}
Returning to our example $\tau$ in $D_{2n}$, consider the strings $(\frac{a_i-1}2,\frac{a_i-3}2,\dots,-\frac{a_i-1}2)$, $1\le i\le 2l$. (These strings give the standard coordinates of $h/2$, where $h$ is the middle element of a Lie triple attached to $\tau$.) There is one way to form a partition $\sigma$ of $n$ such that these strings form the hooks of a left-justified decreasing tableau with shape $\sigma$ and content $i-j$ for the $(i,j)$-box. For example, when $\tau=(3,3,5,5)$ in $D_8$, the partition $\sigma$ is $(3,3,2)$, see Figure \ref{fig}.
\begin{figure}[h]
$$
\begin{tableau}
:.{0}.{1}.{2} \\
:.{{-1}}.{0}.{1} \\
:.{-2}.{{-1}} \\
\end{tableau}$$
\caption{Nilpotent $\tau=(3,3,5,5)$ and partition $\sigma=(3,3,2)$.\label{fig}}
\end{figure}

 Notice that such $\sigma$ always has the property that $\sigma^t=\sigma.$ By \cite{C}, it follows that $\wti\delta^+=\wti\sigma^+_1\oplus\wti\sigma^+_2$ as $\wti W(D_{2n})'$-representations.
\end{proof}

\subsection{Integrality properties of $\Ind_D$}\label{s:integral} Recall the sublattice
$\C Y=\rr(\overline R_\bZ(\bH))\subset \overline R_\bZ(W)$ defined in (\ref{e:Y}). As in the
proof of Corollary \ref{c:indmap}, the restriction $$\Ind_D|_{\C Y}:\C
Y\to \overline R_\bZ(\bH)$$ is an isometric isomorphism and an inverse
of $\rr:\overline R_\bZ(\bH)\to \C Y.$ A natural question is  when
$$\C Y=\overline R_\bZ(W),$$ or equivalently when $\Ind_D$ gives an
isomorphism $\Ind_D: \overline R_\bZ(W)\to \overline R_\bZ(\bH).$

\begin{proposition}\label{p:integral}
Suppose $R$ is irreducible. If one of the following
  conditions is satisfied:
\begin{enumerate}
\item the parameter function $k$ is constant; 
\item $R$ is not simply laced and the parameter $k$ is generic;
\item $R$ is $B_n$ or $G_2$, 
\end{enumerate}
then $\C Y=\overline R_\bZ(W)$.
\end{proposition}

\begin{proof}
(1) Assume without loss of generality that $k\equiv 1.$ Then consider
the restriction of the orthogonal basis $\C B(\overline R_\bZ(\bH))$
from Theorem \ref{t:Honbasis}
to $\overline R_\bZ(W).$ Then $\rr(\C B(\overline R_\bZ(\bH))$ is an
orthogonal set and an $\bR$-basis for $\overline R_\bZ(W).$ To see
that in fact it is a $\bZ$-basis, it is sufficient to recall that the
geometric realization of tempered modules of the Hecke algebra
(\cite{KL,L2}) implies that the matrix of restrictions to $W$ of the set of
irreducible tempered modules with real central character is upper
uni-triangular in the natural ordering given by the closure ordering
of nilpotent orbits, see \cite[section 4]{BM}.

In cases (2) and (3), the basis of $\overline
R_\bZ(\bH)$ constructed in Theorem \ref{t:Honbasis} is in fact an
orthonormal basis given by discrete series modules for generic parameters and
limits of discrete series modules for special parameters. Thus the restriction to $\overline
R_\bZ(W)$ forms an orthonormal $\bZ$-basis as well.

\end{proof}

Proposition \ref{p:integral}(1,2), together with Corollary
\ref{c:indmap}(2) and Theorem \ref{t:ccmap}, allow us to improve the result in Corollary \ref{c:indmap}(4).

\begin{corollary}\label{c:puresupport}
If $\delta\in \overline R_\bZ(W)$ is a rational multiple of a pure element in $\overline R_\bZ(W)$, then $\Ind_D(\delta)$ is supported at  a single central character.
\end{corollary}

\begin{proof}We only need to prove the claim when $R$ is non-simply
  laced and the parameter function $k$ is specialized to a non-generic
  point $k_0$, since otherwise by Proposition \ref{p:integral}(1,2)
  the claim is equivalent with Corollary \ref{c:indmap}(4). Assume this is
  the case and that $X$ is an irreducible $\bH$-module such that
  $\langle\Ind_D(\delta),X\rangle^\EP_\bH\neq 0.$ By Theorem
  \ref{t:ind}(2), this means that $\langle i(\delta),I(X)\rangle_{\wti
  W'}\neq 0,$ and therefore, by Corollary \ref{c:voganchar} the
central character of $X$ equals $\chi^{\wti\delta}$ for some
irreducible component $\wti\delta$ of $i(\delta).$ 

By Theorem
\ref{t:vogan}, we know that $\chi^{\wti\delta}$ depends polynomially
on the parameter $k$. 
Moreover, at generic $k$, $\delta$ is a pure element in $\C Y$ (since
$\C Y=\overline R_\bZ(W)$) and so, by Corollary \ref{c:indmap}(4),
$\chi^{\wti\delta}$ are equal to each other as functions of $k$ for all constituents
$\wti\delta$ of $i(\delta).$  (Here, we implicitly use that Corollary
\ref{c:indmap}(2) gives
$I(\Ind_D(\delta))=i(\delta)$, independently of $k$.) But then the $\chi^{\wti\delta}$ are
equal to each other at $k_0$ as well.

\end{proof}

We end the section with a calculation of $\Ind_D(\delta)$, for the
basis elements $\delta\in \overline R_\bZ(W)$ which are pure, but no
units. These are the cases appearing in $E_7$ and $D_{2n}$ in Theorem \ref{t:Honbasis}.

\begin{example}\label{ex:E7.2} Retain the notation from Example
  \ref{ex:E7.1}. Let
  $\delta=\pi_{A_4+A_1,\triv}|_W$. Then $$\Ind_D(\delta)=2\Phi([\bH\otimes_W(\phi_{512,12}-\phi_{512,11})]).$$
  We remark that $\phi_{512,12}$ and $\phi_{512,11}$ are the
  irreducible Springer representations attached to $(A_4+A_1,\triv)$
  and $(A_4+A_1,\sgn)$, respectively.
\end{example}

\begin{proof}
From Example \ref{ex:E7.1}, we see that $\wti\delta^+=64_s+64_{ss},$
and thus $\Ind_D(\delta)=\Phi([\bH\otimes_W(64_s+64_{ss})\otimes
(S^+-S^-)])=2\Phi([\bH\otimes_W(\phi_{512,12}-\phi_{512,11})]),$ by (\ref{e:tensE7}).
\end{proof}

\begin{example}\label{ex:D2n.2} Retain the notation from Example
  \ref{ex:D2n.1}. Let $\delta=\pi_{(1133),\triv}|_W$ in $\overline R_\bZ(W(D_4)).$ Then 
$$\Ind_D(\delta)=2\Phi([\bH\otimes_W (11\times 2+22\times 0-12\times 1)]),$$
where the notation for irreducible $W(D_4)$-representations is as in
\cite{Ca}. We remark that $12\times 1$ and $11\times 2$  are the irreducible Springer
representations attached to $((1133),\triv)$ and $((1133),\sgn)$, respectively.
\end{example}

\begin{proof} From Example \ref{ex:D2n.1}, we see that 
$$(22\times 0)\otimes S=\wti\sigma_1+\wti\sigma_2,\text{ as }\wti
W(D_4)\text{-representations},$$
where $S$ is the unique spin module for $D_4$ (of dimension $4$), and 
$\wti\sigma_1,\wti\sigma_2$ are two sign self-dual irreducible $\wti
W(D_4)$-representations of dimension $4$. Then
$\wti\sigma_i=\wti\sigma_i^++\wti\sigma_i^-$ as $\wti
W(D_4)'$-representations, and the index of $\pi_{(1133),\triv}$ is the
virtual $\wti W(D_4)'$-representation:
\begin{equation}\label{e:indexD4}
\begin{aligned}
i(\delta)=I(\pi_{(1133),\triv})&=(\wti\sigma_1^+
+\wti\sigma_2^+)-(\wti\sigma_1^- +\wti\sigma_2^-),\text{ while}\\
I(\pi_{(1133),\sgn})&=(\wti\sigma_1^-
+\wti\sigma_2^-)-(\wti\sigma_1^+ +\wti\sigma_2^+).\\
\end{aligned}
\end{equation}
Set $\wti\delta^+=\wti\sigma_1^+ +\wti\sigma_2^+$ and similarly define
$\wti\delta^-$. By definition,
$$\Ind_D(\delta)=\Phi(\bH\otimes_{W(D_4)'} (\wti\delta^+)^*\otimes
(S^+-S^-))=\Phi(\bH\otimes_{W(D_4)}\otimes \Ind_{W(D_4)'}^{W(D_4)} (\wti\delta^-\otimes
(S^+-S^-))  ),$$ where $S^\pm$ are the two irreducible two-dimensional
components of $S|_{\wti W(D_4)'}.$ We need to compute
$\wti\sigma_i^-\otimes S^\pm$; in fact it is sufficient to compute
their induction to $W(D_4).$ Notice that all $\wti\sigma_i^-\otimes
S^\pm$ occur as components of the restriction to $W(D_4)'$ of
$$(22\times 0)\otimes S\otimes S=(22\times 0)\otimes \wedge^\bullet
V=2(12\times 1+11\times 2+22\times 0).$$
Each  of $12\times 1,$ $11\times 2$, $22\times 0$ is sign self-dual,
hence they break into a sum of two irreducible equidimensional
$W(D_4)'$-representations of dimensions $4,$ $3,$ and $1$,
respectively. This means that every
$\Ind_{W(D_4)'}^{W(D_4)}(\wti\sigma_i^\pm\otimes S^\pm)$ equals either
\begin{equation}\label{e:condition}
12\times 1,\text{ or } 11\times 2+22\times 0.
\end{equation}
Using \cite[Lemma 3.8]{CT} and (\ref{e:indexD4}), we see that
$\wti\sigma_1^-+\wti\sigma_2^-$ is contained in  $(12\times
1)\otimes S^+$ and in $(11\times 2)\otimes S^+.$ Combining this with
(\ref{e:condition}), it follows that
$$\Ind_{W(D_4)'}^{W(D_4)}(\wti\sigma_i^-\otimes S^-)=12\times 1 \text{
  and } \Ind_{W(D_4)'}^{W(D_4)}(\wti\sigma_i^-\otimes S^+)=11\times
2+22\times 0,\quad i=1,2,$$
and this proves the claim.

\end{proof}

\section{Dirac induction (analytic version)}\label{s:5}

\subsection{An analytic model}\label{s:5.1}
Let $C^\omega(V)$ be the set of complex analytic functions on
$V.$ Let $M$ be a finite-dimensional (unitary) $W$-module. Following
\cite{HO,EOS} we define a left module action of
$\bH$ on $C^\omega(V)\otimes_\bC M$. Moreover, we define an inner
product on $C^\omega(V)\otimes M$ which makes this (actually the
subset of ``finite'' vectors) into a unitary $\bH$-module, with
respect to the natural $*$-operation on $\bH$ from section \ref{s:3.2}.

We begin by defining certain operators on $C^\omega(V)\otimes M.$

\begin{definition}For every $v\in V$, $f\in C^\omega(V)$, let
  $\partial_v$ denote the directional derivative of $f$ in the
  direction of $v$. Let $Q(v): C^\omega(V)\otimes M\to
  C^\omega(V)\otimes M$ be the operator 
\begin{equation}\label{Qv}
Q(v)(f\otimes
  m)=\partial_v f\otimes m.
\end{equation}
The Weyl group $W$ acts naturally on $C^\omega(V)$ via the left regular
action $(wf)(\xi)=f(w^{-1}\xi).$ For every root $\al\in R$, define the
integral operator $I(\al):C^\omega(V)\to C^\omega(V),$
\begin{equation}\label{Ial}
I(\al)f(\xi)=\int_0^{(\xi,\al^\vee)} f(\xi-t\al)~dt,\quad
\xi\in V.
\end{equation}
For every $\al\in F$, let $Q(s_\al):C^\omega(V)\otimes M\to
C^\omega(V)\otimes M$ be the operator
\begin{equation}\label{Qsal}
Q(s_\al)(f\otimes m)=s_\al f\otimes s_\al m+k_\al I(\al)f\otimes m.
\end{equation}
\end{definition}

\begin{theorem}[{\cite[Theorem 4.11]{EOS}\footnote{In fact, \cite{EOS}
    defines an action of the trigonometric Cherednik algebra  at critical level (which 
    contains $\bH$ as a subalgebra).}}]\label{t:ir} The assignment $v\to Q(v)$,
  $s_\al\to Q(s_\al)$ extends to an action (the ``integral-reflection''
  representation) of $\bH$ on $C^\omega(V)\otimes M.$
\end{theorem}

The explanation is as follows (\cite[section 4.3]{EOS}). Start with the induced module
$\bH\otimes_{\bC[W]} M^*$ with the action of $\bH$ by left
multiplication. As a $\bC$-vector space, $\bH\otimes_{\bC[W]} M^*$ is
isomorphic with $S(V_\bC)\otimes_\bC M^*.$ One traces the action under
this identification. For $\al\in F,$ $v\in V,$ $m'\in M^*$:
\begin{align*}
s_\al(p\otimes m')&=(s_\al\cdot p)\otimes m'=(s_\al(p)
s_\al+k_\al\Delta_\al(p))\otimes m'\\
&=s_\al(p)\otimes s_\al m'+k_\al\Delta_\al(p)\otimes m';\\
v(p\otimes m')&=vp\otimes m',
\end{align*}
where $\Delta_\al$ is the difference operator from (\ref{e:diffop}).

Now consider the dual $(S(V_\bC)\otimes M^*)^*\supset
C^\omega(V)\otimes M.$ The previous action of $\bH$ on
$S(V_\bC)\otimes M^*$ defines an action of $\bH$ on the dual
$(S(V_\bC)\otimes M^*)^*$ by means of the anti-automorphism $\star$
defined on the generators of $\bH$ by $w^\star=w^{-1}$ and
$\xi^\star=\xi.$
To get to the integral-reflection action $Q$ from Theorem \ref{t:ir},
one uses the fact that under the natural pairing
$$(~,~):S(V_\bC)\otimes C^\omega(V)\to \bC,\ (p,f)=(p(\partial)f)(0),$$
the operator $\Delta_\al$ is adjoint to the operator $I(\al).$

\subsection{Unitary structure} 
We generalize the inner product from \cite[(2.6)]{HO}. To begin,
for every Weyl chamber $C$, define an inner product $(~,~)_C$ on
$C^\omega(V)\otimes M$ by extending linearly
\begin{equation}
(\psi_1\otimes m_1,\psi_2\otimes m_2)_C=(\psi_1,\psi_2)_C (m_1,m_2)_M,
\quad \psi_1,\psi_2\in C^\omega(V), m_1,m_2\in M,
\end{equation}
where $(\psi_1,\psi_2)_C=\int_C\psi_1(\eta)\overline{\psi_2(\eta)}d\eta,$
and $(~,~)_M$ is a fixed $W$-invariant inner product on $M$.

Then, for every $f,g\in C^\omega(V)\otimes M,$ set 
\begin{equation}\label{e:unitstr}
(f,g)_{k,C}=\sum_{w\in W} (Q(w)f,Q(w)g)_C.
\end{equation}
The notation is meant to emphasize that this inner product depends on
the multiplicity function $k$ (since $Q$ does) and on the choice of
chamber $C$.
Let $C_+$ be the fundamental Weyl chamber (corresponding to $F$).

\begin{theorem}\label{t:*inv}
The inner product $(~,~)_{k,C_+}$ is $*$-invariant for $\bH.$
\end{theorem}

\begin{proof}
The proof is an adaptation of the proof of \cite[Theorem 2.4]{HO}.

The invariance with respect to $w$ is clear since the inner product
averages over $W$. Fix $\xi\in V.$ A formal argument, using only the
definition of $*$ and the relations in the Hecke algebra shows that
$$(\partial(\xi)f,g)_{k,C}-(f,\partial(\xi)^*g)_{k,C}=\sum_w(\partial(w\xi)Q(w)f,Q(w)g)_C+\sum_w(Q(w)f,\partial(w\xi)Q(w)g)_C.$$

Write $Q(w)f=\sum_j \phi_{w,j}\otimes m_{w,j}$ and
$Q(w)g=\sum_i\psi_{w,i}\otimes n_{w,i}$, where
$\phi_{w,j},\psi_{w,i}\in C^\omega(V)$ and $m_{w,j}, n_{w,i}\in M.$
Set $h_{w,i,j}(\eta)=\phi_{w,j}(\eta)\overline{\psi_{w,i}(\eta)}.$
Then:
\begin{align*}
(\partial(w\xi)Q(w)f,Q(w)g)_C&+(Q(w)f,\partial(w\xi)Q(w)g)_C\\
&=\sum_{i,j}\int_{C}(\partial(w\xi)\phi_{w,j}(\eta))\overline{\psi_{w,i}(\eta)}d\eta
(m_{w,j},n_{w,i})_M\\
&+\sum_{i,j}\phi_{w,j}(\eta)\overline{(\partial(w\xi)\psi_{w,i}(\eta))}d\eta
(m_{w,j},n_{w,i})_M\\
&=\sum_{i,j}\int_C\partial(w\xi)h_{w,i,j}(\eta)d\eta
(m_{w,j},n_{w,i})_M\\
&=\sum_{i,j}\int_{\partial C} h_{w,i,j}(\eta) (w\xi,\nu)d\sigma(\eta) (m_{w,j},n_{w,i})_M,
\end{align*}
by Stokes' theorem, where $\nu$ is the outer normal vector. 
Now assume $C=C_+$.
Since the
boundary of $C_+$ is formed of the intersections with the root
hyperplanes $H_\al$, $\al\in F$, one finds that this equals
\begin{align*}
\sum_{i,j}\sum_{\al\in F}\int_{H_{\al}\cap C_+}
h_{w,i,j}(\eta)(w\xi,\frac{\al^\vee}{|\al^\vee|})d\sigma_\al(\eta) (m_{w,j},n_{w,i}).
\end{align*}
Summing over $w$ and changing the order of summation, one gets:
\begin{align*}
\sum_{\al\in F}\int_{H_\al\cap C_+}(&\sum_{w^{-1}\al>0}\sum_{i,j}
h_{w,i,j}(\eta)(w\xi,\frac{\al^\vee}{|\al^\vee|})(m_{w,j},n_{w,i})+\\
&\sum_{w^{-1}\al<0}\sum_{i,j} h_{w,i,j}(\eta)(w\xi,\frac{\al^\vee}{|\al^\vee|}(m_{w,j},n_{w,i}))d\sigma_\al(\eta).
\end{align*}
We wish to show that this quantity is zero. We make the substitution
$w'=s_\al w$ in the second sum since then
$(w\xi,\frac{\al^\vee}{|\al^\vee|})=-(w'\xi,\frac{\al^\vee}{|\al^\vee|})$ and
$(w')^{-1}\al>0$ if $w^{-1}\al<0.$ But it remains to verify how
$h_{w,i,j}, m_{w,j}, n_{w,i}$ are related to $h_{w',i',j'}, m_{w',j'},
n_{w',i'}$ on the hyperplane $H_\al.$ For this, notice that
$I(\al)h=0$ and $s_\al\cdot h=h$, on $H_\al$ for all $h\in C^\omega(V).$ Therefore, on $H_\al$:
\begin{align}\label{e:crit}
Q(w')f=Q(s_\al)\sum_j\phi_{w,j}\otimes m_{w,j}=\sum_j
s_\al\phi_{w,j}\otimes s_\al m_{w,j}=\sum_j\phi_{w,j}\otimes s_\al m_{w,j},
\end{align}
and similarly for $Q(w')g.$ This implies that on $H_\al,$ we have
$h_{w',i,j}=h_{w,i,j}$, $m_{w',j}=s_\al m_{w,j},$ $n_{w',i}=s_\al
n_{w,i}$ (implicit here is that the sets of indices $(i,j)$ for $w$
and $w'$ are the same). Now the claim follows by the $W$-invariance of
the product $(~,~)_M.$
\end{proof}

\begin{remark}\label{r:negative}
A completely similar argument shows that the inner product $(~,~)_{k,C_-}$ is also invariant, where
$C_-=w_0(C_+)$ is the negative Weyl chamber. 
\end{remark}

Define 
\begin{equation}\label{e:Xmodel}
\C X_\omega(M)=\{f\in C^\omega(V)\otimes M:
(\partial(p)f,\partial(p)f)_{k,C_+}<\infty,\text{ for all }p\in S(V_\bC)\}.
\end{equation}
 Using the $W$-invariance of $(~,~)_{k,C_+}$, it is easy to see that the action of $\bH$ preserves $\C X_\omega(M).$ Thus we have:

\begin{corollary}The $\bH$-module $\C X_\omega(M)$ is $*$-pre-unitary.
\end{corollary}

\subsection{Global Dirac operators} We wish to define Dirac operators on
the spaces $\C X_\omega(M)$.
 For this we need to trace again through the action of $\bH$ and
the chain of identifications after Theorems \ref{t:ir}.

Let $S$ be a  spin module for the Clifford algebra $C(V)$ 
and let $E$ be a genuine $\wti W$-module. Then $E\otimes S$ is a $W$-representation, so we have the $\bH$-module $\C X_\omega(E\otimes S).$

If $\CB$
is an orthonormal basis of $V$, define $D\in
\End_\bH(\bH\otimes_{W}(E\otimes S))$ via
\begin{equation}
D(h\otimes x\otimes y)=\sum_{\xi\in \CB}h\wti\xi\otimes x\otimes
\xi y,\quad h\in \bH, x\in E, y\in S.
\end{equation}
The definition does not depend on the choice of basis $\CB,$ and
moreover $D$ is well-defined:
\begin{align*}
D(hw\otimes &w^{-1}(x\otimes y))=\sum_{\xi\in\CB}hw\wti\xi\otimes
{\wti w}^{-1} x\otimes \xi{\wti w}^{-1} y\quad\text{(for some pullback
  $\wti w$ of $w$ in $\wti W$)}\\
&=\sum_{\xi\in \CB}h\wti{w(\xi)}w\otimes {\wti w}^{-1}x\otimes {\wti
  w}^{-1} w(\xi) y   \quad\text{ (where } \wti{w(\xi)}=w\cdot\wti\xi
\cdot w^{-1}\text{)}\\
&=\sum_{\xi\in \CB} h\wti{w(\xi)}w\otimes w^{-1}(x\otimes w(\xi)y)=\sum_{\xi\in w(\CB)} h\wti\xi\otimes (x\otimes \xi y)=D(h\otimes
x\otimes y),
\end{align*}
since $w(\CB)$ is also an orthonormal basis of $V$.

Clearly $D$ commutes with the module action of $\bH$, since the $\bH$-action is by left multiplication.

In the identification $\bH\otimes_{W}(E\otimes S)\cong S(V_\bC)\otimes (E\otimes S)$, $D$ acts as
follows:
\begin{align*}
D(p\otimes (x\otimes y)&=\sum_{\xi\in\CB}p\wti \xi\otimes (x\otimes
\xi y)=\sum_{\xi\in\CB} (p\xi\otimes (x\otimes \xi y)-p T_\xi\otimes
(x\otimes\xi y)\\
&=\sum_{\xi\in\CB} \xi p\otimes (x\otimes\xi
y)-\sum_{\xi\in\CB}p\otimes T_\xi(x\otimes \xi y),
\end{align*}
where $T_\xi$ were defined in (\ref{e:Txi}).
The formula for $D^2$ can be computed analogously with the one in the local case (\ref{e:Dsquare}).

\begin{proposition}\label{p:globalDsquare} As operators on $\bH\otimes_W(E\otimes S)$, we have
\begin{equation}\label{e:globalDsquare}
D^2=-\Omega\otimes 1\otimes 1+1\otimes \Omega_{\wti W}\otimes 1.
\end{equation}
\end{proposition}

\begin{proof}
The proof is a completely analogous calculation to that in the proof of \cite[Theorem 3.5]{BCT}. For simplicity, denote
\begin{equation}
R^2_{\circ}=\{(\al,\beta)\in R\times R: \al,\beta>0, \al\neq \beta, s_\al(\beta)<0\}.
\end{equation}
Let $\{\xi_i\}$ be an orthonormal basis of $V$ and let $p\otimes (x\otimes y)$ be an element of $\bH\otimes (E\otimes S).$ Then we have
\begin{align*}
D^2(p\otimes (x\otimes y))&=\sum_{i,j}p \wti\xi_i\wti\xi_j\otimes (x\otimes\xi_j\xi_i y)\\
&=\sum_i p \wti\xi_i^2\otimes x\otimes \xi_i^2y+\sum_{i<j} p[\wti\xi_i,\wti\xi_j]\otimes (x\otimes \xi_j\xi_i y)
\end{align*}
Now, we use the identities
\begin{equation}
\sum_i\wti\xi_i^2=\Omega-\frac 14\sum_{\al>0}\langle\al^\vee,\al^\vee\rangle-\frac 14\sum_{(\al,\beta)\in R^2_0}k_\al k_\beta\langle\al^\vee,\beta^\vee\rangle s_\al s_\beta,\ \text{(see \cite[Theorem 2.11]{BCT})}
\end{equation}
and 
\begin{equation}
[\wti\xi_i,\wti\xi_j]=[T_{\xi_j},T_{\xi_i}]=\frac 14\sum_{(\al,\beta)\in R^2_{\circ}}k_\al k_\beta((\al^\vee,\xi_j)(\beta^\vee,\xi_i)-(\beta^\vee,\xi_j)(\al^\vee,\xi_i))s_\al s_\beta, 
\end{equation}
(see for example \cite[Lemma 2.9]{BCT}). It follows that:
\begin{align*}
D^2&(p\otimes(x\otimes y))=-\Omega p\otimes (x\otimes y)+\frac 14\sum_{\al>0}\langle\al^\vee,\al^\vee\rangle p\otimes (x\otimes y)\\
&+\frac 14\sum_{(\al,\beta)\in R^2_0}k_\al k_\beta\langle\al^\vee,\beta^\vee\rangle p\otimes s_\al s_\beta(x\otimes y)\\
&+\frac 14\sum_{i<j} \sum_{(\al,\beta)\in R^2_{\circ}}k_\al k_\beta p\otimes ((\al^\vee,\xi_j)(\beta^\vee,\xi_i)-(\beta^\vee,\xi_j)(\al^\vee,\xi_i))s_\al s_\beta (x\otimes \xi_j\xi_i y)\\
&=-\Omega p\otimes (x\otimes y)+\frac 14\sum_{\al>0}\langle\al^\vee,\al^\vee\rangle+\frac 14\sum_{(\al,\beta)\in R^2_0}k_\al k_\beta\langle\al^\vee,\beta^\vee\rangle p\otimes s_\al s_\beta(x\otimes y)\\
&-\frac 14\sum_{i\neq j} \sum_{(\al,\beta)\in R^2_{\circ}}k_\al k_\beta p\otimes (\beta^\vee,\xi_j)(\al^\vee,\xi_i)s_\al s_\beta (x\otimes \xi_j\xi_i y),
\end{align*}
where we used again that $\xi_i\xi_j=-\xi_j\xi_i$ for $i\neq j.$ Now changing the order of summation, the last double sum becomes
\begin{equation*}
\frac 14\sum_{(\al,\beta)\in R^2_\circ} k_\al k_\beta p\otimes s_\al s_\beta (x\otimes (\beta^\vee\al^\vee+\langle\al^\vee,\beta^\vee\rangle )y),
\end{equation*} 
where we identify $\al^\vee,\beta^\vee$, via the inner product $\langle~,~\rangle,$ as elements of $C(V)$. Finally, making the necessary cancellations, we arrive at the desired formula.

\end{proof}

Dualizing and making the identifications, we arrive at the following
definition.

\begin{definition}\label{d:globaldirac}
Let $E$ be a genuine $\wti W$-module. The Dirac operator
$D_E\in\End_{\bH}(C^\omega(V)\otimes (E\otimes S))$ is given by
\begin{equation}
D_E(f\otimes x\otimes y)=\sum_{\xi\in\CB}\partial(\xi)f\otimes
x\otimes\xi y-\sum_{\xi\in\CB} f\otimes T_\xi(x\otimes \xi y),
\end{equation}
where $T_\xi=\frac 12\sum_{\beta\in R^+} k_\beta
(\xi,\beta^\vee) s_\beta.$
Clearly, $D_E$ preserves the finite vectors in $C^\omega(V)\otimes (E\otimes S)$, and thus it defines an operator $D_E\in \End_{\bH}(\C X_\omega(E\otimes S)).$
\end{definition}

\begin{example} We consider the Hecke algebra $\bH$ for $sl(2)$. Here
  $V=\bR\al$, $W=\bZ/2\bZ$, and $\bH$ is generated by $s$ and
  $\xi\in V$ subject to
$$s\cdot\xi+\xi\cdot s=\langle\xi,\al^\vee\rangle.$$
We assume that the inner product on $V$ is normalized so that the
length of $\al^\vee$ is $\sqrt 2.$ 

If we make the identification $\xi\to \xi\al,$ $\xi\in \bR$, we
may regard the functions $f$ as $f:\bR\to \bR$, and the action of the
Hecke algebra on $C^\omega(\bR)\otimes M$ is:
\begin{align*}
&Q(s)(f(\xi)\otimes m)=f(-\xi)\otimes (s\cdot m)+(\int_{-\xi}^\xi f(t)
dt)\otimes m;\\
&Q(\xi)f=\frac {df}{d\xi}.
\end{align*}

The cover $\wti W$ is isomorphic to $\bZ/4\bZ$ and there are two
genuine $\wti W$-types (both spin modules), $\chi_+$ and $\chi_-$ given
by multiplication by $i$ and $-i$ respectively. Fix $S=\chi_+$ and
take the basis $\CB=\{\frac 1{\sqrt 2}\al^\vee\}$. 

\begin{enumerate}
\item $E=\chi_+$. Note that $E\otimes S$ is the $\sgn$
  $W$-representation. We have $D_+(f\otimes x\otimes y)=i(\frac
  {df}{d\xi}\otimes x\otimes y+f\otimes x\otimes y).$ Then $\ker
  D_+=\bR e^{-\xi}\otimes\sgn$, and one checks that this is the
  Steinberg module. It is unitary with respect to the inner product $(~,~)_{C_+}$.
\item $E=\chi_-$. Note that $E\otimes S$ is the $\triv$
  $W$-representation. We have $D_-(f\otimes x\otimes y)=i(\frac
  {df}{d\xi}\otimes x\otimes y-f\otimes x\otimes y).$ Then $\ker
  D_-=\bR e^{\xi}\otimes\triv$, and one checks that this is the
 trivial module. It is unitary with respect to the inner product $(~,~)_{C_-}.$
\end{enumerate}
\end{example}

\begin{definition}
Assume that $\dim V$ is odd, and let $S^+,S^-$ be the two spin modules
of $C(V).$ Define $D^\pm_E: \C X_\omega(E\otimes S^\pm)\to \C
X_\omega(E\otimes S^\mp)$ by composing the Dirac operators
$D^\pm_E\in\End_\bH(\C X(E\otimes S^\pm))$ from Definition
\ref{d:globaldirac} with the vector space isomorphism $S^+\to S^-$, as
in the local case in section \ref{s:3.4}. 
\end{definition}

Now assume that $\dim V$ is even. Let $S^+, S^-$ be the two spin modules of $C(V)_\even$ and let $E$ be a genuine representation of $\wti W'.$ Then $E\otimes S^\pm$ are $W'$-representations, so we can consider the $\bH$-modules 
\begin{equation}
\C X'_\omega(E\otimes S^\pm):=\C X_\omega(\bC[W]\otimes_{W'}S^\pm).
\end{equation}
 Define Dirac operators $$D^\pm: \bH\otimes_{W'}(E\otimes S^\pm)\to \bH\otimes_{W'}(E\otimes S^\mp)$$ by
\begin{equation*}
D^\pm (h\otimes x\otimes y)=\sum_{\xi\in \CB}h\wti\xi\otimes x\otimes
\xi y,\quad h\in \bH, x\in E, y\in S.
\end{equation*}
Again, $D^\pm$ are  well-defined, independent of the choice of basis $\CB,$ and commute with the left action of $\bH$. The formula for $D^\pm D^\mp$ is the same as in the formula (\ref{e:globalDsquare}):
\begin{equation}\label{e:globalDsquare2}
D^\pm D^\mp=-\Omega\otimes 1\otimes 1+1\otimes \Omega_{\wti W}\otimes 1,
\end{equation}
 because in that calculation, the elements of $W$ that occur are actually in $W'$, so they can still be moved across the tensor product.

 In the identification 
\begin{equation}
\bH\otimes_{W'}(E\otimes S^\pm)\cong \bH\otimes_W(\bC[W]\otimes_{W'}(E\otimes S^\pm))\cong S(V_\bC)\otimes (\bC[W]\otimes_{W'}(E\otimes S^\pm)),
\end{equation}
 the action of $D^\pm$ is as follows:
\begin{equation*}
\begin{aligned}
D^\pm(p\otimes(w\otimes(x\otimes y)))&=\sum_{\xi\in\CB} \xi p\otimes (w\otimes (x\otimes\xi
y))-\sum_{\xi\in\CB}p\otimes (T_\xi w\otimes (x\otimes \xi y)).
\end{aligned}
\end{equation*}
Dualizing, we obtain the following definition.

\begin{definition} Assume that $\dim V$ is even. Let $E$ be a genuine $\wti W'$-module. The Dirac operators $D^\pm_E: \C X_\omega'(E\otimes S^\pm)\to \C X_\omega'(E\otimes S^\mp)$ are given by
\begin{equation}
\begin{aligned}
D^\pm_E(f\otimes (w\otimes(x\otimes y)))&=\sum_{\xi\in\CB} \partial(\xi)f\otimes (w\otimes (x\otimes \xi y))-\sum_{\xi\in\CB} f\otimes (T_\xi w\otimes (x\otimes \xi y)).
\end{aligned}
\end{equation}
\end{definition}

\subsection{Global Dirac index}
With these definitions at hand, we can define the global Dirac index. For uniformity of notation, set
$\C X'_{\omega}(M)=\C X_\omega(M)$, when $\dim V$ is odd as well, for every $W'=W$-module $M$. For every genuine $\wti W'$-representation $E$, we have defined the Dirac operators
\begin{equation}
D^\pm_E: \C X'_\omega(E\otimes S^\pm)\to \C X'_\omega(E\otimes S^\mp),
\end{equation}
which commute with the action of $\bH$.
As in the case of local Dirac operators, it is easy to see that $D^+_E$ is adjoint to $D^-_E$ with respect to the unitary structure of $ \C X'_\omega(E\otimes S^\pm)$.

\begin{comment}
\begin{definition}\label{d:globalindex}
For a given genuine $\wti W'$-representation $E$, the global Dirac index is formally
\begin{equation}
I_E=\ker D^+_E-\ker D^-_E,
\end{equation}
a virtual $\bH$-module.
\end{definition}

We explain what we mean by this definition.
\end{comment}

If $\lambda\in W\backslash V^\vee_\bC$, define 
\begin{equation}
\C X'_\omega(M)_\lambda=\{f\in \C X'_\omega(M): \partial(p)f=p(\lambda)f,\text{ for all }p\in S(V_\bC)^W\}.
\end{equation}
A classical result of Steinberg implies that $\C X'_\omega(M)_\lambda$ is finite-dimensional. Since $D_E^\pm$ commute with the $\bH$-action, we have the restricted Dirac operators
\begin{equation}
D_E^\pm(\lambda):\C X'_\omega(E\otimes S^\pm)_\lambda\to \C X'_\omega(E\otimes S^\mp)_\lambda,
\end{equation}
and the restricted Dirac index
\begin{equation}\label{e:restrindex}
I_E(\lambda)=\ker D_E^+(\lambda)-\ker D_E^-(\lambda),
\end{equation}
a virtual $\bH$-module. We will see that $I_E(\lambda)=0$ except for finitely many values of $\lambda$, depending on $E$. We have the following easy restriction.

\begin{lemma}\label{l:nonzeroker} Assume $\wti\delta$ is an irreducible $\wti W'$-representation. Then
$\ker D_{\wti \delta}^\pm(\lambda)\neq 0$ only if $\langle\lambda,\lambda\rangle=\langle\chi^{\wti\delta},\chi^{\wti\delta}\rangle$, where $\chi^{\wti\delta}$ is the central character attached to $\wti\delta$ by Definition \ref{d:voganchar}.
\end{lemma}

\begin{proof}
Assume that $\ker D_{\wti\delta}^+(\lambda)\neq 0.$ Then $\ker D_{\wti\delta}^-D_{\wti\delta}^+(\lambda)\neq 0$, and by (\ref{e:globalDsquare}) and (\ref{e:globalDsquare2}), we get that $\langle\lambda,\lambda\rangle=\wti\delta(\Omega_{\wti W})=\langle\chi^{\wti\delta},\chi^{\wti\delta}\rangle.$
\end{proof}

In the next subsection, we show that in fact, under the assumptions of Lemma \ref{l:nonzeroker}, $\lambda=\chi^{\wti\delta}$, as a consequence of  Theorem  \ref{t:vogan}.

%The following lemma is the analogue of Lemma \ref{l:CT} with the same proof.

\begin{lemma}\label{l:analyticindex}
For every $\lambda\in W\backslash V^\vee_\bC,$ $I_E(\lambda)=\C X'_\omega(E\otimes S^+)_\lambda-\C X'_\omega(E\otimes S^-)_\lambda.$
\end{lemma}

\begin{proof}
We have $(x,D_E(\lambda)^-y)_{k,C_+}=(D^+_E(\lambda)x,y)_{k,C_+}=0,$ whenever $x\in\ker D^+_E(\lambda),$ which shows that $\ker D^+_E(\lambda)\subset (\im D^-_E(\lambda))^\perp.$ Conversely,  if $x\in (\im D^-_E(\lambda))^\perp\subset \C X'_\omega(E\otimes S^+)_\lambda,$ then $0=(x,D^-_E(\lambda)y)_{k,C_+}=(D^+_E(\lambda)x,y)_{k,C_+},$ for all $y\in \C X'_\omega(E\otimes S^-)_\lambda.$ Specializing to $y=D^+_E(\lambda)x$, we see that $D^+_E(\lambda)x=0,$ hence $x\in\ker D^+_E(\lambda).$ Therefore, we proved 
\begin{align*}
&\C X'_\omega(E\otimes S^+)_\lambda=\ker D^+_E(\lambda)\oplus \im D^-_E(\lambda),\text{ and similarly, } \\
&\C X'_\omega(E\otimes S^-)_\lambda=\ker D^-_E(\lambda)\oplus \im D^+_E(\lambda).
\end{align*}
Then $D^+_E(\lambda)$ maps $\im D^-_E(\lambda)$ onto $\im D^+_E(\lambda)$ and $D^-_E(\lambda)$ maps $\im D^-_E(\lambda)$ onto $\im D^+_E(\lambda)$, and they induce an isomorphism $\im D^+_E(\lambda)\cong \im D^-_E(\lambda)$ as $\wti W'$-representations. The claim follows.

\end{proof}

\begin{definition}
\label{d:globalindex}
For a given genuine $\wti W'$-representation $E$, we define the global Dirac index to be
\begin{equation}\label{eq:sumindex}
I_E=\bigoplus_{\lambda} I_E(\lambda),
\end{equation}
where $I_E(\lambda)$ is as in (\ref{e:restrindex}). In fact, we conjecture that
$\ker D_E^+$ and $\ker D_E^-$ are finite-dimensional, and as a result,
that $I_E=\ker D_E^+-\ker D_E^-.$
\end{definition}

\subsection{A realization of discrete series modules}\label{s:ds}

\begin{lemma}\label{l:ds1}
Let $(\pi,X)$ be an irreducible $\bH$-module and $M$ a finite-dimensional $W'$-module such that $\Hom_\bH(X,\C X'_\omega(M))\neq 0.$ Then $(\pi,X)$ is a discrete series $\bH$-module.
\end{lemma}

\begin{proof}
Let $x\in X$ be an $S(V_\bC)$-weight vector in $X$ with weight $\nu\in V^\vee.$ This means that for all $\xi\in V_\bC$, $\pi(\xi)x=(\xi,\nu)x.$ Let $\kappa\in \Hom_\bH(X,\C X'_\omega(M))$ be nonzero. Then $$\partial(\xi)\kappa(x)=\kappa(\pi(\xi)x)=(\xi,\nu)x.$$
Recall that in $ \C X'_\omega(M)$, $\partial(\xi)$ acts by differentiation only on the function part of the tensor product, so we may assume without loss of generality, that $\kappa(x)=f\otimes m,$ for some $f\in C^\omega(V)$, $m\in M$. This means that the one-dimensional vector space generated by $f$ is invariant under differentiation by all elements $p\in S(V_\bC)$, and therefore $f$ is an exponential function $f(v)=ce^{(v,\nu)}$, $v\in V.$ 

But since $f\otimes m\in   \C X'_\omega(M)$, $f$ must be an $L^2$-function on $C_+$, and therefore the weight $\nu$ must satisfy the strict Casselman criterion (see \ref{bHtemp}).
\end{proof}

\begin{lemma}\label{l:ds2}
Suppose that $(\pi,X)$ is an irreducible discrete series $\bH$-module and $M$ a finite-dimensional $W'$-module. Then 
\begin{equation}
\Hom_\bH(X,\C X'_\omega(M))\cong \Hom_{W'}(M^*,X^*).
\end{equation}
\end{lemma}

\begin{proof}
Since the dual of $S(V_\bC)$ can be identified with Laurent series on $V_\bC$, the proof of Lemma \ref{l:ds1} shows that every nontrivial homomorphism of $X$ into $(\bH\otimes_{W'}M^*)^*$, under the assumption that $X$ be a discrete series module, lands in fact in $\C X'_\omega(M).$ In other words, we have
$$\Hom_\bH(X,\C X'_\omega(M))\cong\Hom_\bH(X,(\bH\otimes_{W'} M^*)^*).$$
Furthermore, using the tautological isomorphism $\Hom(A,B^*)\cong \Hom(B,A^*),$ we find that
\begin{equation}
\begin{aligned}
\Hom_\bH(X,\C X'_\omega(M))&\cong\Hom_\bH(\bH\otimes_{W'}M^*,X^*)\\
&\cong \Hom_{W'}(M^*,X^*),
\end{aligned}
\end{equation}
where the last step is the Frobenius isomorphism.
\end{proof}

Recall that by Corollary \ref{c:EPds}, if $(\pi,X)$ is a discrete series module, then $\rr(X)$ is a unit element in $\overline R_\bZ(W)$, where $\rr$ is the restriction map from \ref{e:rmap}. Therefore, by Theorem \ref{t:CT}(2), there exist unique irreducible $\wti W'$-representations $\wti\delta^+_X,\wti\delta^-_X$ such that
\begin{equation}
I(X)=\wti\delta_X^+-\wti\delta_X^-,
\end{equation}
and moreover, by Corollary \ref{c:voganchar}, the central character of $X$ is $\chi^{\wti\delta_X^+}=\chi^{\wti\delta_X^-}$, where the notation is as in Definition \ref{d:voganchar}.

\begin{theorem}\label{t:globalindex}For any irreducible genuine $\wti W'$-module $E$, $I_E$ (as defined by (\ref{eq:sumindex})) is   
a virtual finite linear combination of discrete series representations.
Let $(\pi,X)$ be an irreducible discrete series module with central character
$\chi^{\wti \delta_X^+}$ and let $E$ be an irreducible $\wti W'$-representation. Then 
\begin{equation}\label{homeq}
 \Hom_\bH(X,I_E)\cong\Hom_\bH(X,I_E(\Lambda(\wti\delta_X))) \cong \Hom_{\wti W'}(E^*,I(X)),
\end{equation}
and the dimension of these spaces is $\begin{cases}1,&\text{ if }E^*\cong \wti\delta^+_X\\ -1, &\text{ if }E^*\cong\wti\delta^-_X\\ 0, &\text{ otherwise.}\end{cases}$.
In particular, $I_{\wti\delta^+_X}$ is the unique irreducible discrete series $\bH$-module $X$ that has Dirac index $\wti\delta^+_X-\wti\delta^-_X.$

\end{theorem}

\begin{proof}
Since the central character of $X$ is $\Lambda(\wti\delta_X)$, we necessarily have $\Hom_\bH(X,I_E)\cong\Hom_\bH(X,I_E(\Lambda(\wti\delta_X))).$ By Lemma \ref{l:analyticindex}, this space is isomorphic with $$\Hom_\bH(X,\C X'_\omega(E\otimes S^+)-\C X'_\omega(E\otimes S^-)),$$
which by Lemma \ref{l:ds2} equals $$\Hom_{W'}(E^*\otimes (S^+-S^-)^*,X^*)=\Hom_{\wti W'}(E^*,X\otimes (S^+-S^-))=\Hom_{\wti W'}(E^*,I(X));$$
here, we used that 
%$E^*\cong E$ as $\wti W'$-modules, and 
$X^*\cong X$ as $W'$-modules, while the last equality is from Lemma \ref{l:CT}. Now (\ref{homeq}) follows from this using that $I(X)=\wti\delta_X^+-\wti\delta_X^-.$
 
To complete the proof, notice that Theorem \ref{t:CT}(2) combined with Corollary \ref{c:EPds} implies that if $X\not\cong Y$ are two distinct irreducible discrete series modules, then $\wti\delta_X^\pm\neq \wti\delta_Y^\pm,$ and therefore $X$ is the only discrete series that contributes to $I_{\wti\delta_X^+}.$

\end{proof}

\begin{remark}
If instead of the inner product $(~,~)_{k,C_+}$, one uses $(~,~)_{k,C_-}$ (see Remark \ref{r:negative}), then one can realize all irreducible anti-discrete series of $\bH$ (i.e., the Iwahori-Matsumoto duals of the discrete series) in the corresponding analytic models and indices of Dirac operators. All the results are the obvious analogues.
\end{remark}

\begin{remark}Most of the constructions here apply equally well to noncrystallographic root systems.  Theorem \ref{t:globalindex} shows that if $X$ is an irreducible discrete module such that $X|_W$ is a unit element in $\overline R_\bZ(W)$ with respect to the elliptic pairing, then $X$ is isomorphic to the global Dirac index for an irreducible $\wti W$-representation, and in particular, it is unitary. 
\end{remark}

\begin{comment}
For simplicity, we assume that $X$ is unitary (this will be the case of interest for us). From (\ref{Dadj}), we have $(x,D^-y)_{X\otimes S^+}=(D^+x,y)_{X\otimes S^-}=0,$ whenever $x\in\ker D^+,$ which shows that $\ker D^+\subset (\im D^-)^\perp.$ Conversely,  if $x\in (\im D^-)^\perp,$ then $0=(x,D^-y)_{X\otimes S^+}=(D^+x,y)_{X\otimes S^-},$ for all $y\in X\otimes S^-.$ Specializing to $y=D^+x$, we see that $D^+x=0,$ hence $x\in\ker D^+.$ Therefore, we have proved $$X\otimes S^+=\ker D^+\oplus \im D^-,\text{ and similarly, } X\otimes S^-=\ker D^-\oplus \im D^+.$$
Then $D^+$ maps $\im D^-$ onto $\im D^+$ and $D^-$ maps $\im D^-$ onto $\im D^+$, and they induce an isomorphism $\im D^+\cong \im D^-$ as $\wti W'$-representations. The claim follows.
\end{comment}

%%%%%%%%%%%%%%%%%%%%%%%%%%%%%%%%%%%%%%%%%%%%%%%%%%%%%%%%%%

\ifx\undefined\bysame
\newcommand{\bysame}{\leavevmode\hbox to3em{\hrulefill}\,}
\fi

\end{document}